\documentclass[11pt]{amsart}

\usepackage{amsmath,amsthm}
\usepackage{amssymb,amsfonts}
\usepackage{hyperref}
\usepackage{tikz}
\usepackage{enumerate}
\usepackage{graphicx}
\usepackage{calligra}
\usepackage{stmaryrd}
\DeclareSymbolFont{bbold}{U}{bbold}{m}{n}
\DeclareSymbolFontAlphabet{\mathbbm}{bbold}

\title{Embedding \cstar-algebras into the Calkin algebra}

\theoremstyle{plain}
	\newtheorem{theorem*}{Theorem}
	\newtheorem{theorem}{Theorem}[section]
	\newtheorem{thrm}{Theorem}[section]
	
	\newtheorem{proposition}[theorem]{Proposition}
	\newtheorem{lemma}[theorem]{Lemma}
	
	\newtheorem{corollary}[theorem]{Corollary}
	\newtheorem{crlr}[thrm]{Corollary}

	\newtheorem{claim}{Claim}[theorem]

\theoremstyle{definition}
	\newtheorem*{NP*}{Na{\u\i}mark's Problem}
	\newtheorem*{dia*}{The Diamond Principle ($\diamondsuit$)}
	\newtheorem{definition}[theorem]{Definition}

	\newtheorem{conjecture}[theorem]{Conjecture}
	\newtheorem{question}[theorem]{Question}
	
\theoremstyle{remark}

\DeclareMathOperator{\id}{id}

\DeclareMathOperator{\Ad}{Ad}
\DeclareMathOperator{\stone}{St}
\DeclareMathOperator{\clopen}{Clop}

\newcommand{\C}{\mathbb{C}}

\newcommand{\N}{\mathbb{N}}
\newcommand{\Q}{\mathbb{Q}}

\newcommand{\set}[1]{\{#1\}}					

\newcommand{\cst}{\mathrm{C}^*}
\newcommand{\cstar}{$\mathrm{C}^*$}
\newcommand{\cQ}{\mathcal Q} 
\newcommand{\cD}{\mathcal D} 
\newcommand{\cB}{\mathcal B} 
\newcommand{\cU}{\mathcal U} 
\newcommand{\cH}{\mathcal H} 
\newcommand{\cP}{\mathcal P} 
\newcommand{\bbN}{\mathbb N} 
\newcommand{\bbP}{\mathbb P} 
\newcommand{\bbC}{\mathbb C} 
\newcommand{\bbL}{\mathbb L} 
\newcommand{\bbQD}{\mathbb {QD}} 
\newcommand{\bbQ}{\mathbb Q}
\newcommand{\bbE}{\mathbb E}
\newcommand{\bbNN}{\mathbb N^{\mathbb N}} 
\DeclareMathOperator{\Fin}{Fin} 
\DeclareMathOperator{\Bf}{\mathcal{B}_f}
\renewcommand{\phi}{\varphi} 
\DeclareMathOperator{\Id}{Id}

\DeclareMathOperator{\supp}{supp}

\newcounter{my_enumerate_counter}
\newcommand{\pushcounter}{\setcounter{my_enumerate_counter}{\value{enumi}}}
\newcommand{\popcounter}{\setcounter{enumi}{\value{my_enumerate_counter}}}

\newcommand{\Addresses}{{
  \bigskip
  \footnotesize
  
  I.~Farah, \textsc{Department of Mathematics and Statistics, York University, 4700 Keelee Street,
    North York, Ontario, Canada, M3J 1P3}\par\nopagebreak
  \textit{E-mail address}: \texttt{ifarah@yorku.ca}\par\nopagebreak
  \textit{URL}: \texttt{http://www.math.yorku.ca/~ifarah/}

    \medskip
  
  G.~Katsimpas, \textsc{Department of Mathematics and Statistics, York University, 4700 Keelee Street,
    North York, Ontario, Canada, M3J 1P3}\par\nopagebreak
  \textit{E-mail address}: \texttt{gkats@mathstat.yorku.ca}
  \par\nopagebreak
  \textit{URL}: \texttt{http://sites.google.com/view/georgioskatsimpas}
 
  \medskip
  
  A.~Vaccaro, \textsc{Department of Mathematics, University of Pisa, Largo Bruno Pontecorvo 5
    Pisa, Italy, 56127 - Department of Mathematics and Statistics, York University, 4700 Keelee Street,
    North York, Ontario, Canada, M3J 1P3}\par\nopagebreak
  \textit{E-mail address}: \texttt{vaccaro@mail.dm.unipi.it}\par\nopagebreak
  \textit{URL}: \texttt{http://people.dm.unipi.it/vaccaro/index.html}
  }}

\usepackage[mode=multiuser,status=draft,lang=english]{fixme}
\fxsetup{theme=colorsig}
\hypersetup{colorlinks=false, pdfborder={0 0 0}}

\usepackage{titlesec}
\titlelabel{\thetitle. \;}
\titleformat*{\section}{\centering\large\scshape}

\author{Ilijas Farah,
Georgios Katsimpas,
Andrea Vaccaro}
\date{}

\thanks{Partially supported by IF's NSERC grant.}

\begin{document}
\maketitle

\begin{abstract}
We prove that, under Martin's Axiom, 
every \cstar-algebra of density character less than continuum embeds into the
Calkin algebra. Furthermore, we show that it is consistent with ZFC that there is a \cstar-algebra of
density character less than continuum that does not embed into the Calkin algebra.
\end{abstract}

\keywords

\section{Introduction} \label{sctn1}
The Calkin algebra $\cQ(H)$ is the quotient of $\cB(H)$,
the algebra of bounded linear operators on a complex, separable, infinite-dimensional Hilbert space $H$,
modulo the ideal of the compact operators $\mathcal{K}(H)$.
It is considered to be the noncommutative analogue of the Boolean algebra $\cP(\bbN)/\Fin$\footnote{$\Fin$ 
denotes the ideal of all finite subsets of $\bbN$, also known as the Fr\'echet ideal.} (see e.g., \cite{farahicm} and 
\cite{We:Set}) and,
as a consequence, results about  $\cP(\bbN)/\Fin$ often translate into questions (frequently nontrivial) about $\cQ(H)$. 
In this note we study the analogue of the question ``Which linear orderings embed into $\cP(\bbN)/\Fin$?''. 
In order to put our study into the proper context, we start by reviewing some known results about the latter problem. 

Note that  $\cP(\bbN)$ embeds into $\cP(\bbN)/\Fin$.  To define an embedding, send $A\subseteq \bbN$ 
to the equivalence class of the set $\{(2n+1)2^m: n \in \N, \ m\in A\}$. 
Every countable linear ordering~$\bbL$ embeds into $\cP(\bbN)$, and therefore into $\cP(\bbN)/\Fin$. 
One way to see this is to enumerate
the elements of $\bbL$ as $a_n$, for $n\in \bbN$, and define $\Phi\colon \bbL\to \cP(\bbN)$ 
by $\Phi(a_m)=\{n: a_n\leq a_m\}$.

There is a simple characterization of linear orderings $\bbL$ that embed into $\cP(\bbN)$. 
A linear ordering $\bbL$ embeds into $\cP(\bbN)$ if and only if it has a countable 
subset $\{a_n: n\in \bbN\}$ which is separating in the sense that for all $x<y$ in $\bbL$ there exists
$n$ such that $x\leq a_n<y$ or $x<a_n\leq y$.\footnote{This condition is strictly weaker than being separable.} 
To prove  the  direct implication, given    $\{a_n: n\in \bbN\}$, 
one can define $\Phi$ as above. The converse implication is straightforward.  
No such characterization exists for the class of linear orderings that embed into $\cP(\bbN)/\Fin$.  

Since $\cP(\bbN)/\Fin$ 
is a countably saturated atomless Boolean algebra, all linear orderings of cardinality $\aleph_1$ embed into $\cP(\bbN)/\Fin$. 
Thus the Continuum Hypothesis, CH, implies that a linear order embeds into $\cP(\bbN)/\Fin$ if and only if its cardinality 
is at most $2^{\aleph_0}$. By  \cite{laver1979linear}, if ZFC is consistent\footnote{G\"odel's Incompleteness Theorem 
implies that it is not possible to prove the consistency (i.e. the existence of a model)
of ZFC within ZFC, unless ZFC is inconsistent.}
 then 
the assertion that all linear orderings of cardinality at most $2^{\aleph_0}$ embed into $\cP(\bbN)/\Fin$ 
is relatively consistent with ZFC plus the negation of CH.
Laver's model is however an exception, and in some models of ZFC (if there are any!)  
the class of linear orderings which embed into $\cP(\bbN)/\Fin$ can be downright bizarre. This class is also very important. For example, 
Woodin's condition for the automatic continuity of  Banach algebra homomorphisms from $C([0,1])$
asserts that if there exists a discontinuous homomorphism from $C([0,1])$ into a Banach algebra then 
a nontrivial initial segment of an ultrapower $\bbNN/\cU$ embeds into $\cP(\bbN)/\Fin$
(\cite{DaWo:Introduction}).\footnote{This is usually stated in terms of embedding into the directed 
set $(\bbNN,\leq^*)$, but a linear order embeds into $(\bbNN,\leq^*)$ if and only if it embeds
into $\cP(\bbN)/\Fin$; see e.g., \cite[Proposition~0.1]{Fa:Embedding} or \cite[Lemma~3.2]{woodin1984discontinuous}.}
Every $*$-homomorphism between \cstar-algebras is automatically continuous, and all homomorphisms between \cstar-algebras are continuous in Woodin's model. It is not known whether it is provable in ZFC that every homomorphism between \cstar-algebras with dense range is continuous 
(see the introduction to \cite{ozawa2006invitation}). 

The question of what linear orderings embed into the poset of projections of the Calkin algebra or into 
the poset of self-adjoint elements of the Calkin algebra may be of an independent interest. 
However, the question that we consider here is strictly operator-algebraic: Which \cstar-algebras embed into the 
Calkin algebra? 
This is also a non-commutative analogue of the question of which abelian \cstar-algebras 
embed into $\ell_\infty/c_0$. By the Gelfand--Naimark duality, this corresponds to asking 
which compact Hausdorff spaces are continuous images of $\beta\bbN\setminus \bbN$, 
the \v{C}ech--Stone remainder of $\bbN$. By Parovi\v cenko's Theorem, 
having weight not greater than  $\aleph_1$ is a sufficient condition (alternatively, this can be proved 
  by elementary model theory;  see the discussion in \cite[p. 1820]{DoHa:Universal}). 
However, the situation in ZFC 
is quite nontrivial (\cite{DoHa:Lebesgue}, 
\cite{DoHa:Images}).

The analogue of the cardinality of a \cstar-algebra (or a topological space)  $A$
 is the density character. 
It is defined as the least cardinality of a dense subset of $A$. 
Thus the \cstar-algebras of density character $\aleph_0$ are exactly the 
separable \cstar-algebras. The density character of a nonseparable \cstar-algebra is equal to the 
minimal cardinality of a generating subset and also to 
 the minimal cardinality of a dense $(\bbQ+i\bbQ)$-subalgebra. 
Every separable \cstar-algebra embeds into $\cB(H)$ and therefore into $\cQ(H)$, by a standard amplification argument. In addition,
all \cstar-algebras of density character $\aleph_1$ embed into $\cQ(H)$, but the proof 
is surprisingly nontrivial (\cite{farah2017calkin})
due to the failure of countable saturation in 
the Calkin algebra (\cite[\S 4]{FaHa:Countable}). Since the density character of $\cQ(H)$ is $2^{\aleph_0}$, 
\cstar-algebras of larger density character do not embed into $\cQ(H)$ and once again
 CH gives the simplest possible characterization of the class of \cstar-algebras that embed into $\cQ(H)$. 
 In this note we make the next step and we investigate what happens when CH fails,
 focusing on \cstar-algebras of density character strictly less than $2^{\aleph_0}$.
 
 \begin{thrm} \label{indep}
The assertion 
`Every \cstar-algebra of density character strictly less than~$2^{\aleph_0}$ embeds into 
the Calkin algebra' is independent from ZFC. It is moreover independent from 
ZFC+$2^{\aleph_0}=\aleph_3$, and $\aleph_3$ is the minimal cardinal with this property. 
\end{thrm} 

The most involved part in the proof of Theorem \ref{indep} is showing that the
statement `All \cstar-algebras of density character strictly less than~$2^{\aleph_0}$ embed into 
$\cQ(H)$' is consistent with ZFC+$2^{\aleph_0}>\aleph_2$. This will be achieved via Theorem \ref{T1} (which is proved in \S \ref{sctn4}) using forcing.

The method of forcing was introduced by Cohen to prove the independence of CH from ZFC,
and later developed to deal with more general independence phenomena
(see \S \ref{sctn2.2}).
The \emph{countable chain condition} (or ccc) is a property of forcing notions that ensures no cardinals or 
cofinalities are collapsed, and all stationary sets are preserved,  in the forcing extension (see Definition \ref{cccMA}).  

\begin{thrm} \label{T1} For every \cstar-algebra $A$ there exists a ccc forcing notion $\mathbb{E}_A$ which 
forces that $A$ embeds into $\cQ(H)$. 
\end{thrm} 
Rephrasing the statement of Theorem \ref{T1}, every \cstar-algebra, regardless of its density character, 
can be embedded into the Calkin algebra in a forcing extension of the universe
obtained without collapsing any cardinals or cofinalities.

The following Corollary (proved as  Corollary~\ref{MA}) is the consistency result needed to prove Theorem \ref{indep} and  follows from the proof of Theorem \ref{T1}.

\begin{crlr}\label{C.MA}
Assume Martin's Axiom, MA. Then, every \cstar-algebra with density character strictly less than $2^{\aleph_0}$ embeds into the Calkin algebra.
\end{crlr}

In the case when the continuum is not greater than $\aleph_2$, the conclusion 
of Corollary~\ref{C.MA} follows from \cite{farah2017calkin}. 
A combination of this corollary with results from \cite{Andrea.PhD} yields the proof of Theorem
\ref{indep}.

\begin{proof}[Proof of Theorem \ref{indep}]
As pointed out above, if the cardinality of the continuum is not greater than $\aleph_2$
then all \cstar-algebras of density character strictly less than $2^{\aleph_0}$ embed into the Calkin algebra. 

Martin's Axiom is relatively consistent with the continuum being equal to $\aleph_3$  (\cite[Theorem V.4.1]{kunen}) and by Corollary~\ref{C.MA} in this model all \cstar-algebras of density character not greater than $\aleph_2$
 embed into the Calkin algebra. 

On the other hand, in a model obtained by adding $\aleph_3$ Cohen reals to a model of CH
we get $2^{\aleph_0} = \aleph_3$  and
the Calkin algebra has no chains of projections of order type $\aleph_2$. 
This was proved in \cite[Section~2.5]{Andrea.PhD} by adapting a well-known argument from Kunen's PhD thesis (\cite[Section 12]{kunenthesis}). 
Therefore in this model the abelian \cstar-algebra $C(\aleph_2+1)$ (where the ordinal $\aleph_2+1$ is 
endowed with the order topology) does not embed into $\cQ(H)$. 
\end{proof}

We remark that Theorem~\ref{T1} was inspired by an analogous
fact holding for partial orders and $\cP(\bbN)/\Fin$: For every partial order $\mathbb{P}$
there is a ccc forcing notion which forces the existence of an embedding of
$\mathbb{P}$ into $\cP(\bbN)/\Fin$. While the proof
of this latter fact is an  elementary exercise, the proof of Theorem~\ref{T1} is fairly sophisticated,
and will take most of this paper.   
At a critical place it makes use of some variations of 
Voiculescu's theorem (\cite[Corollary 1.7.5]{brownozawa}; see Theorem \ref{lemma:2} and Corollary \ref{voic}).

The paper is organized as follows: In section \ref{sctn2} we introduce the notation, some
basic notions and preliminary lemmas needed later in the paper. Section \ref{sctn3} discusses  two special cases of Theorem \ref{T1} as a warmup: 
The case when $A$ is abelian and the case when it is quasidiagonal. 
Only the latter case requires novel ideas. 
 Section \ref{sctn4}
is where the partial order $\mathbb{E}_A$ is defined and where Theorem \ref{T1} is proved. Finally, section \ref{sctn5} is devoted to questions and
concluding remarks.

\section{Preliminaries and notation} \label{sctn2}
\subsection{\cstar-algebras} \label{sctn2.1}
By $H$ we will always denote the complex, separable, infinite-dimensional Hilbert space
$\ell_2(\mathbb{N})$ and by $\mathcal{B}(H)$ the space of linear, bounded operators on $H$. The space of all finite-rank operators on $H$ is denoted $\Bf(H)$. Its norm-closure, denoted $\mathcal{K}(H)$, is the ideal of compact operators. The notation $\mathcal{U}(H)$ is reserved for the group of unitary operators on $H$.
The Calkin algebra $\mathcal{Q}(H)$ is the quotient of $\mathcal{B}(H)$ by the compact operators and for what follows $\pi :\mathcal{B}(H)\rightarrow \mathcal{Q}(H)$ will always denote the quotient map.
For $h\in\Bf(H) $, $h^+$ denotes the orthogonal projection onto its range and $h^-$ is the projection onto the space of 1-eigenvectors of $h$ (i.e. the space of all vectors
$\xi$ such that $h \xi = \xi$). We write ${\Bf(H)}^{\leq 1}_{+}$ for the collection of all finite-rank positive contractions on $H$.
An operator $T \in \mathcal{B}(H)$ is \emph{way above} $S$, $T \gg S$ in symbols,
if $TS = S$. For two projections $P, Q$ we have $P \ll Q$ iff $P \le Q$.
We write $T {\sim}_{\mathcal{K}(H)}S$ and say that $T$ and $S$ \textit{agree modulo the compacts} to indicate that $T-S\in\mathcal{K}(H)$. Similarly, given a \cstar-algebra $A$, two maps ${\varphi}_1:A\rightarrow\mathcal{B}(H)$ and ${\varphi}_2:A\rightarrow\mathcal{B}(H)$ are said to agree modulo the compacts if ${\varphi}_1(a) {\sim}_{\mathcal{K}(H)} {\varphi}_2(a)$ for every $a\in A$. A net of operators $\set{T_i}_{i \in I}$ \emph{strongly converges} to an operator $T$
if for each $\xi \in H$ the net $\set{T_i\xi}_{i \in I}$ converges to $T\xi$. We remark that
to verify the strong convergence of a net it suffices to check it on a dense subset of $H$.

Given two vectors $\xi$ and $\eta$ of a normed vector space and $\epsilon >0$, the notation $\xi \approx_\epsilon \eta$ stands for $\lVert\xi - \eta\rVert<\epsilon$. 
We abbreviate `$F$ is a finite subset of $A$' as $F\Subset A$. 
If   $F$ is a subset of a \cstar-algebra then  $C^\ast(F)$ denotes 
the \cstar-algebra  generated by $F$. If $A$ is unital and $u\in A$ is a unitary element, then $\Ad u$ denotes the automorphism of $A$ which sends $a$ to $uau^*.$ A representation $\Phi :A\rightarrow \mathcal{B}(H)$ is called \textit{essential} if $\Phi(a) \in \mathcal{K}(H)$ implies $\Phi(a) = 0$ for all $a \in A$. Note that all (non-zero) representations of unital, simple, infinite-dimensional \cstar-algebras on $H$ are faithful (i.e. injective) and essential.
A unital, injective $*$-homomorphism $\Theta:A\rightarrow \mathcal{Q}(H)$ is \textit{trivial} if there exists a unital (and necessarily essential) representation $\Phi:A\rightarrow\mathcal{B}(H)$ such that $\pi\circ\Phi = \Theta$ and, in this case, the map $\Phi$ is called a \textit{lift} of $\Theta$. Moreover, $\Theta$ is called \textit{locally trivial} if its restriction to any unital separable \cstar-subalgebra of $A$ is trivial.

Mainly for convenience, in the proof of Theorem \ref{T1} in section \ref{sctn4} we shall exclusively be  concerned with embeddings of unital and simple  \cstar-algebras into the Calkin algebra, as any unital $*$-homomorphism from a unital and simple  \cstar-algebra into $\mathcal{Q}(H)$ is automatically injective. This causes no loss of generality, as a result of the next proposition.

\begin{proposition}[{\cite[Lemma 2.1]{farah2017calkin}}]\label{fvlemma}
Every  \cstar-algebra $A$ embeds into  a unital and simple  \cstar-algebra $B$ of the same density character
as~$A$. \qed
\end{proposition}

The following standard consequence of Voiculescu's theorem 
 will be invoked frequently throughout the rest of this manuscript.

\begin{theorem}[{\cite[Corollary 1.7.5]{brownozawa}}] \label{lemma:2} 
Let $A$ be a unital, separable \cstar-algebra and let  $\Phi :A\rightarrow \mathcal{B}(H)$ and $\Psi:A\rightarrow \mathcal{B}(H)$ be two faithful, essential, unital representations.
Then, for every $F \Subset A$ and $\epsilon >0$
there exists a unitary $u\in\mathcal{U}(H)$ such that:
\begin{enumerate}
\item The maps $\Ad u\circ\Phi$ and $\Psi$ agree modulo the compacts. \label{item1}
\item $\lVert \Ad u\circ\Phi(a) - \Psi(a)\rVert <\epsilon$ for all $a\in F$.\qed
\end{enumerate}
\end{theorem} 
See also \cite{arveson} and \cite[Section 3]{higsonroe} for a detailed proof of the
theorem above.
We will also be using the next variant, which allows  to find a unitary as in item
\ref{item1} of the previous theorem which
in addition is equal to the identity on a given finite-dimensional space:

\begin{corollary}\label{voic}
Let $A$ be a unital, separable \cstar-algebra and consider two faithful, essential, unital representations $\Phi :A\rightarrow \mathcal{B}(H)$ and $\Psi:A\rightarrow \mathcal{B}(H)$.
Then,  for every $F\Subset A $ and every finite-dimensional subspace $K\subseteq H$ there exists a unitary $w\in\mathcal{U}(H)$ such that:
\begin{enumerate}
\item The maps $\Ad w\circ\Phi$ and $\Psi$ agree modulo the compacts.
\item $\Ad {w\circ\Phi(a)}(\xi) = {\Phi(a)}(\xi)$ for every $a\in F$ and $\xi\in K$.
\end{enumerate}
In particular, the set 
\[
Z=\{\Ad w\circ \Phi: w\in \mathcal{U}(H),\ \Ad w \circ \Phi(a) \sim_{K(H)} \Psi(a) \ \text{for all}\  a\in A\}
\]
has $\Phi$ in its closure with respect to strong convergence.
\end{corollary}

\begin{proof}
Let $F\Subset A$, $K\subseteq H$ be a finite-dimensional subspace  and we let $P\in\mathcal{B}(H)$ be the orthogonal projection onto $K$.
By Theorem \ref{lemma:2},  we can find a unitary $v\in\mathcal{U}(H)$
such that $\Ad v\circ \Phi$ and $\Psi$ agree modulo the compacts.
Let $Q$ be the finite-rank projection onto the subspace spanned by the set 
$K\cup\{ \Phi(a)K : a \in F\}$ and let $w\in\mathcal{U}(H)$ be a finite-rank modification of $v$ such that $wQ=Qw=Q$. 
Then $\Ad w \circ \Phi$ and $\Ad v \circ \Phi$ agree modulo the compacts  and 
$(\Ad w \circ \Phi)(a) P = \Phi(a) P $  for all $a\in F$. 
\end{proof}
 
The following lemma will be invoked  for proving  a density result (Proposition \ref{prop:1}).

\begin{lemma} \label{lemma:1}
Let $T\in \mathcal{B}(H)$ be a finite-rank projection. For every $\epsilon>0$ there exists $\delta >0$ such that if $S\in \mathcal{B}(H)$ and $\lVert T-S\rVert <\delta$, then there is a unitary $u\in\mathcal{U}(H)$ satisfying the following:
\begin{enumerate}
\item $uT[H]\subseteq S[H]$, namely the image space of $uT$ is contained in the image space
of $S$,
\item $\lVert (u-\Id_{H})T\rVert<\epsilon$,
\item $u-\Id_{H}\in\Bf(H)$,
\item \label{i4} for every orthogonal projection $P$ onto a subspace of $T[H]$ such that $SP = P$, we have that
$uP = P$ holds.
\end{enumerate}
\end{lemma}

\begin{proof}
Let $\set{\xi_1,\dots, \xi_k}$ be an orthonormal basis of the space
of all eigenvectors of $S$ whose eigenvalue is 1 and which are moreover contained in $T[H]$.
Fix $\{\xi_1,\dots,\xi_n\}$ an orthonormal basis of $T[H]$
extending $\set{\xi_1,\dots, \xi_k}$. If $\lVert T-S\rVert < \delta < 1$,
the set $\{S\xi_1,\dots,S\xi_n\}$ (which linearly spans $ST[H]$) is linearly
independent. In fact, if $\xi \in T[H]$ has norm one and is such that $S\xi = 0$, then
$\lVert T\xi \rVert = \lVert \xi \rVert < \delta$, which is a contradiction.
Applying the Gram-Schmidt process to $\{S\xi_1,\dots,S\xi_n\}$
we obtain an orthonormal basis $\{\eta_1,\dots,\eta_n \}$ for $ST[H]$,
which for sufficiently small choice of $\delta$ (which depends on the dimension of $T[H]$) is such that
\[
\lVert \xi_i - \eta_i\rVert <\frac{\epsilon}{n} , \ i=1,\dots,n.
\]
Denote by $V$ the finite-dimensional space spanned by $T[H]$ and $ ST[H]$.
Let $\{\xi_1,\ldots,\xi_m\}$ be an orthonormal basis of $V$
that extends $\{\xi_1,\ldots,\xi_n\}$ and, similarly, $\{\eta_1,\ldots,\eta_m\}$
an orthonormal basis of $V$ extending $\{\eta_1,\ldots,\eta_n\}.$
This naturally defines a unitary $w:V\rightarrow V$ by sending the
vector $\xi_i$ to $\eta_i$ for every $i=1,\ldots,m.$ Finally, define
$u\in\mathcal{U}(H)$ to be equal to $w$ on $V$ and equal to
the identity on the orthogonal complement of $V$. The unitary $u$ satisfies the desired properties,
in particular item \ref{i4} of the statement holds since $\eta_i = \xi_i$ for $i \le k$ by our initial
choice of $\set{\xi_1,\dots, \xi_k}$, orthonormal basis of the space of all eigenvectors of $S$
of eigenvalue 1 in $T[H]$.
\end{proof}

\subsection{Set Theory and Forcing} \label{sctn2.2}\label{S.Why}
As stated in the introduction, Theorem \ref{T1} is an application of the method of forcing.
For a standard introduction to this topic see \cite{kunen}; see also 
\cite{DaWo:Introduction} and \cite{weaver2014forcing}.

We start with some technical definitions. 
Two elements $p,q$ of a partial order (or poset) $(\mathbb{P},\le)$ are \emph{compatible} if there exists
$s\in P$ such that $s\le p$ and $s\le q$. Otherwise, $p$ and $q$ are
\emph{incompatible}. A subset $A \subseteq \mathbb{P}$ is an \emph{antichain}
if its elements are pairwise incompatible. 
A subset $D \subseteq \mathbb{P}$
is \emph{dense} if for every $p \in \mathbb{P}$ there is $q \in D$ such that $q \le p$.
A subset $D$ of $\mathbb{P}$ is \emph{open} if it is closed downwards,
i.e. $p \in D$ and $q \le p$ implies $q \in D$.
A non-empty subset $G$ of $\mathbb{P}$
is a \emph{filter} if $q \in G$ and $q \le p$ implies
$p \in G$, and if for any $p, q \in G$ there exists $r \in G$ such that $r \le p$, $r \le q$.
Given a family $\mathcal{D}$ of dense open subsets of $\mathbb{P}$, a filter $G$ is $\mathcal{D}$\emph{-generic}
if it has non-empty intersection with each element of $\mathcal{D}$.

A  \emph{forcing notion} (or \emph{forcing}) is a partially ordered set (poset), whose elements
are called \emph{conditions}.
Naively, the forcing method produces, starting from a poset $\mathbb{P}$,
an extension of von Neumann's universe $V$. The extension is obtained by adding to $V$ a
filter $G$ of $\mathbb{P}$ which intersects \emph{all} dense open subsets of $\mathbb{P}$.
This generic extension, usually denoted by $V[G]$,
is a model of ZFC, and its theory depends  on combinatorial properties of 
  $\mathbb{P}$ and (to some extent) on the choice of $G$.
A  condition $p\in \bbP$  \emph{forces} a sentence $\varphi$ in the language of ZFC 
if $\varphi$ is true in $V[G]$ whenever $G$ is a generic filter containing $p$. If $\phi$ is true
in every generic extension $V[G]$, we say that $\bbP$ \emph{forces} $\phi$.

Unless $\bbP$ is trivial, no   filter intersects every dense open subset
of $\bbP$.
For this reason, 
the forcing method is combined with a L\"owenheim--Skolem reflection argument 
and applied to countable models of ZFC. 
If $M$ is a countable model of ZFC and $\bbP\in M$, then the existence of an $M$-generic filter $G$ (i.e. intersecting every open dense subset of $\mathbb{P}$ in $M$) of $\bbP$
is guaranteed by the Baire Category Theorem (\cite[Lemma III.3.14]{kunen})\footnote{For metamathematical reasons related to 
G\"odel's Incompleteness Theorem, 
one usually considers models of a large enough finite fragment of ZFC. By other 
metamathematical considerations, for all practical purposes this issue can be safely ignored;
see \cite[Section IV.5.1]{kunen}.}.


An obvious method for embedding a given \cstar-algebra $A$ into the Calkin algebra is to 
generically add a bijection between a dense subset of $A$ and $\aleph_0$ (i.e. to `collapse' the density 
character of $A$ to $\aleph_0$). The completion of $A$ in the forcing extension (routinely identified with $A$) 
 is then separable
and therefore embeds into the Calkin algebra of the extension.  
However, if the density character of $A$ is collapsed, then 
this  results in a \cstar-algebra that has little to do with the original algebra $A$. We shall give two examples. 

Fix an uncountable cardinal $\kappa$. 
If $A$ is $\cst_r(F_{\kappa})$, the reduced group algebra of the free group 
with $\kappa$ generators, then collapsing $\kappa$ to $\aleph_0$ 
makes $A$  isomorphic to $\cst_r(F_{\aleph_0})$ (better known as $\cst_r(F_\infty)$). It 
is  not difficult to  prove that, if a cardinal $\kappa$ is not collapsed,
then the completion of $\cst_r(F_{\kappa})$ in the extension is isomorphic to $\cst_r(F_\kappa)$ as computed
in the extension. This is not automatic as, for example, the completion of 
the ground model Calkin algebra in a forcing extension will rarely be isomorphic to the Calkin algebra in the extension.

 A more drastic example is provided  by the $2^\kappa$ nonisomorphic 
\cstar-algebras each of which is an inductive limit of full matrix algebras of the form $M_{2^n}(\bbC)$ for $n\in \bbN$ constructed in \cite[Theorem~1.2]{FaKa:NonseparableII}. After collapsing $\kappa$ to $\aleph_0$, 
all of these \cstar-algebras become isomorphic to the CAR algebra. This is because
it can be proved that the $K$-groups of $A$ are invariant under forcing and, by Glimm's classification result,
unital and separable inductive limits of full matrix algebras are isomorphic   (e.g. \cite{blackadar}). 
A similar effect can be produced even with a forcing that  
preserves cardinals if it collapses a  stationary set  (\cite[Proposition~6.6]{FaKa:NonseparableII}).

Instead of `collapsing' the cardinality of $A$, our approach is to `inflate' the Calkin algebra. 
More precisely, we prove that Martin's Axiom implies that the Calkin algebra has already been `inflated'.

\emph{Forcing axioms} are far-reaching extensions of the Baire Category Theorem that
enable one to apply forcing without worrying about metamathematical issues. 
  Corollary \ref{C.MA} will be proved by applying   Martin's axiom,  
   the simplest (and most popular) 
  forcing axiom. 

\begin{definition} \label{cccMA} A  poset $(\mathbb{P},\le)$ satisfies the
\emph{countable chain condition}
(or \emph{ccc}) if every antichain in $\bbP$ is at most countable.

\emph{Martin's Axiom}, \emph{MA}, asserts that for every ccc poset $\bbP$ and every family~$\cD$ of fewer than $2^{\aleph_0}$ 
dense open subsets of $\bbP$, there exists a $\cD$-generic filter in~$\bbP$.
\end{definition}
It is relatively consistent with ZFC that Martin's axiom holds and the continuum
is larger than any prescribed cardinal $\kappa$ (\cite[Theorem V.4.1]{kunen}). 
The countable chain condition is  the single most flexible property of forcing notions that enables one to iterate forcing and obtain forcing
 extensions with various prescribed properties  (see e.g. \cite[Theorem~IV.3.4]{kunen}). 
Our posets will have the following strong form of ccc.  
A poset $(\mathbb{P},\leq)$ has \textit{property~K} if every uncountable subset of $\mathbb{P}$
contains a further uncountable subset in which any two elements are compatible.

The proof strategy in section \ref{sctn4} is as follows.
Given a \cstar-algebra $A$,
we start by defining a forcing notion $\mathbb{E}_A$ (Definition \ref{def:1})
whose generic filters (if any) allow to build an embedding of $A$ into $\cQ(H)$ (Proposition \ref{genemb}).
We then proceed to show
that $\mathbb{E}_A$ is ccc (Proposition \ref{propK}), and that the existence of sufficiently generic filters inducing the existence of an embedding of $A$ into $\cQ(H)$ is guaranteed in models of ZFC + MA (Corollary \ref{MA}).

The following lemma will be used when proving that a given forcing 
notion is ccc. A family $\mathcal{C}$ of sets forms a \emph{$\Delta$-system} with \emph{root} $R$
if $X\cap Y=R$ for any two distinct sets $X$ and $Y$ in $\mathcal{C}$. When the sets in $\mathcal{C}$ are pairwise disjoint, one obtains the special case with $R=\emptyset$.  

\begin{lemma}[$\Delta$-System Lemma, {\cite[Lemma III.2.6]{kunen}}]\label{delta}
Every uncountable family of finite sets contains an uncountable $\Delta$-system. \qed
\end{lemma}

\section{The Cases of Abelian and Quasidiagonal \cstar-algebras}\label{sctn3}
In this section, we discuss two special cases  of Theorem \ref{T1}, those
corresponding to  the classes of abelian and quasidiagonal \cstar-algebras. Their proofs (the first of which 
is standard) 
are intended to provide intuition and demonstrate the increase in complexity regarding the corresponding forcing notions that are implemented. It  also displays the natural progression behind Theorem \ref{T1}. We will omit most of the technical
details in this section,
as the results discussed here can be easily inferred by the proofs of the subsequent parts
of the paper.
The reader eager to transition right away to the proof of Theorem \ref{T1} can safely skip ahead to section \ref{sctn4}.

\subsection{Embedding Abelian \cstar-algebras into $\ell_\infty /c_0$}\label{sctn3.1}
The main focus in this part will be on obtaining the abelian version of Theorem \ref{T1}:

\begin{proposition}\label{abelianemb}
For every abelian \cstar-algebra $A$ there exists a ccc forcing notion which forces that $A$ embeds into ${\ell}_{\infty}/c_0.$
\end{proposition}
Exploiting the fact that the categories of Boolean algebras, Stone spaces (i.e. zero-dimensional, compact, Hausdorff spaces) and \cstar-algebras of continuous functions on Stone spaces are all equivalent  (by a combination of the Stone duality  \cite[section II.4]{johnstone} and  the Gelfand--Naimark duality  \cite[section IV.4]{johnstone}), one can translate the statement of the proposition above to a statement regarding Boolean algebras. In particular, it is enough to show that for any Boolean algebra $B$ there exists a ccc forcing notion which forces that $B$ embeds into $\mathcal{P}(\N)/\Fin.$  If $B$ is a Boolean algebra, we denote by $\stone(B)$ its \textit{Stone space}, the space of all ultrafilters on $B$ equipped with the Stone topology.

To see the aforementioned translation, first of all note that it suffices to prove the assertion of Proposition \ref{abelianemb} for \cstar-algebras of the form $C(Y)$ with $Y$ being a Stone space, as every abelian \cstar-algebra embeds into such an algebra. Indeed, any
abelian \cstar-algebra $C(X)$ naturally embeds into the von Neumann algebra
$L^\infty(X)$ which, being a real rank zero unital \cstar-algebra, is of the form $C(Y)$ with $Y$
zero-dimensional, compact and Hausdorff. We provide an alternative proof
for the reader who is not familiar with the theory of von Neumann algebras.
Every non-unital, abelian \cstar-algebra embeds into its unitization, which is a \cstar-algebra of continuous functions on a compact, Hausdorff space $X$.
For any compact, Hausdorff space $X$, let $X_d$ consist of the underlying set of $X$ equipped with the discrete topology. Then, the identity map from $X_d$ to $X$ uniquely extends to a continuous map from $\beta X_d$ onto $X$ and this, in turn,  implies the existence of an embedding of $C(X)$ into $C(\beta X_d).$ The \v{C}ech--Stone compactification of a discrete space is always zero-dimensional  and this establishes the previous claim.

Now, if $X$ is a Stone space, consider the Boolean algebra $B = \clopen(X)$ of all clopen subsets of $X$. Due to the Stone duality, the existence of a ccc forcing notion that forces the embedding of $B$ into  $\mathcal{P}(\N)/\text{Fin}$ yields (in any generic extension of the universe) a continuous surjection  from $\stone(\mathcal{P}(\N)/\text{Fin})\cong \beta\N\setminus\N$ onto $\stone(B)\cong X.$ By contravariance due to the Gelfand--Naimark duality, one obtains  an injective $*$-homomorphism from $C(X)$ into $C(\beta\N\setminus\N)$, with the latter being isomorphic to $ \ell_{\infty}/c_0.$

Thus, we turn our attention to providing the forcing notion guaranteed by the following folklore proposition:

\begin{proposition}\label{boolean}
For every Boolean algebra $B$ there exists a ccc forcing notion $\mathbb{P}_B$ which forces that $B$ embeds into $\mathcal{P}(\N)/\text{Fin}.$
\end{proposition}

We identify the subsets of $\N$ with their characteristic functions, and we think them as elements
of $2^\N$. With this in mind, we view the Boolean algebra $\mathcal{P}(\N)/\text{Fin}$ as the space of all binary sequences $2^\N$ modulo the equivalence relation
\[
x\sim y \hspace{8pt} \text{if and only if} \hspace{8pt} |\{n\in\N : x(n)\neq y(n)\}| < \aleph_0
\]
for all $x,y\in 2^\N.$ 

\begin{definition}\label{booleandef1}
Fix a Boolean algebra $B$ and let $\mathbb{P}_B$ be the set of all triples
\[
p = (B_p,n_p,\psi_p)
\]
where:
\begin{enumerate}
\item $B_p$ is a finite  Boolean subalgebra of $B$, 
\item $n_p\in\N$,
\item$\psi_p:B_p\rightarrow 2^{n_p}$ is an arbitrary map.
\pushcounter
\end{enumerate}
For $p,q\in\mathbb{P}_B$, we say that \textit{$p$ extends $q$} and write  $p < q$ if  the following hold:
\begin{enumerate}
\popcounter
\item $B_q\subseteq B_p$,
\item $n_q < n_p$,
\item $\psi_q\subset \psi_p$  (i.e. $\psi_p(a)(i) = \psi_q(a)(i)$ for all $a\in B_q$ and $i\leq n_q$),
\item the map from $B_q$ into   $2^{n_p-n_q}$ given by 
\begin{align*}
a &\mapsto   \psi_p(a)_{\restriction[n_q, n_p)}
\end{align*}
is an injective homomorphism of Boolean algebras.
\end{enumerate}
\end{definition}
This  defines a strict partial order on $\mathbb{P}_B$.  Conditions in $\mathbb{P}_B$ represent partial maps from a finite subset of $B$ to an initial segment of a characteristic function corresponding to a subset of $\N$.
Any finite Boolean subalgebra of $B$ is isomorphic to the Boolean algebra 
given by the powerset of a finite set and hence can be embedded into $2^m$ for $m\in\N$ large
enough. Therefore one can always extend a given condition $p \in \mathbb{P}_B$ to a $q < p$
such that $B_q$ contains any arbitrary finite subset of $B$ and $n_q > n_p$,
while making sure that in the added segment the map
is actually an injective homomorphism. For this reason,
a generic filter $G$ in $\mathbb{P}_B$ provides a pool of maps
which can be `glued' together in a coherent way, inducing thus a function $\Psi_G$ which, by genericity, is defined everywhere on $B$:
\begin{align*}
\Psi_G: B &\to \mathcal{P}(\N)  \\
b & \mapsto \bigcup_{\set{p \in G : b \in B_p}} \psi_p(b).
\end{align*}
Here we identify $\psi_p(b) \in 2^{n_p}$ with the corresponding subset of $n_p$.
Moreover, by definition of the order relation on $\mathbb{P}_B$, the map $\Psi_G$ is,
modulo the ideal of finite sets, injective and preserves all Boolean operations.

By using a standard uniformization argument and an application of the $\Delta$-System Lemma (Lemma \ref{delta}), when given an uncountable set of conditions $U \subseteq \mathbb{P}_B$,
it is possible to find an uncountable $W \subseteq U$, $n \in \N$ and $Z \Subset B$
such that $n_p = n$, $B_p \cap B_q = Z$ and $\psi_p(b) = \psi_q(b)$ for all $p,q \in W$ and $b \in
Z$. Thus the problem of whether $\mathbb{P}_B$ is ccc is reduced to the following:

\begin{lemma}
Let $p,q\in \mathbb{P}_B$ be two conditions such that $n_p=n_q$ and the maps $\psi_p,\psi_q$ agree on $B_p\cap B_q$. Then, $p$ and $q$ are compatible.
\end{lemma}

To see that this holds, define $B_s$ to be the (finite) Boolean subalgebra of $B$ that is generated by $B_p\cup B_q$ and choose a Boolean algebra isomorphism $f:B_s\rightarrow 2^m$ for some $m\in\N$. Set $n_s=n_p+m$ and define the map $\psi_s$ to be equal to $\psi_p$ concatenated with $f$ on $B_p$, equal to $\psi_q$ concatenated with $f$ on $B_q\setminus B_p$ and equal to zero elsewhere. Then, the condition $s=(B_s,n_s,\psi_s)$ extends both $p$ and $q$.

\subsection{Embedding Quasidiagonal \cstar-algebras into $\mathcal{Q}(H)$}\label{sctn3.2}
Quasidiagonal \cstar-algebras possess  strong local properties and can be thought (at least in the separable case) as consisting of compact pertubations of simultaneously block-diagonalisable operators. A map $\varphi:A\rightarrow B$ between unital  \cstar-algebras is called \textit{unital completely positive} (abbreviated as \textit{u.c.p.}) if it is unital, linear and the tensor product map $\phi\otimes \Id_n : A\otimes M_n(\C)\rightarrow B\otimes M_n(\C)$ defined on matrix algebras over $A$ and $B$ is positive for 
all $n\in\N$ (\cite{blackadar}, section II.6.9). U.c.p. maps are always contractive and $\ast$-preserving. For a \cstar-algebra $A$, we will denote its unitization by $\tilde{A}$.

\begin{definition}\label{uqddef}
A \cstar-algebra $A$ is \textit{quasidiagonal} if  for every finite set $F\Subset \tilde{A}$ and  $\epsilon >0,$ there exist $n\in\N$ and a u.c.p. map $\varphi:\tilde{A}\rightarrow M_n(\C)$ such that
\[
\lVert \varphi(ab)-\varphi(a)\varphi(b)\rVert < \epsilon \ \text{for all} \ a,b\in F
\]
and
\[
\lVert \varphi(a)\rVert > \lVert a\rVert -\epsilon \ \text{for all} \ a\in F.
\]
\end{definition}

This section is devoted to the following:

\begin{proposition}\label{qdt1}
For every quasidiagonal \cstar-algebra $A$ there exists a ccc poset ${\mathbb{Q}\mathbb{D}}_A$ which forces an embedding of $A$ into $\mathcal{Q}(H).$
\end{proposition}

As opposed to the proof of Theorem \ref{T1} in section \ref{sctn4}, where we can apply Proposition
\ref{fvlemma}, we will not assume that
$A$ is simple in the proof of Proposition \ref{qdt1}. Such assumption would have made Definition \ref{qddef} slightly simpler, but, to our knowledge, it is not known whether
it is possible to embed a given quasidiagonal \cstar-algebra into a simple quasidiagonal one
(an application of the Downward L\"owenheim--Skolem Theorem (\cite[Theorem 2.6.2]{Muenster}) would then provide a quasidiagonal simple \cstar-algebra with the same density character as the one we started with).
We may assume though that $A$ is unital. Fix ${\{e_n\}}_{n\in\N}$ an orthonormal basis of $H$ and for every $n\in\N$ let $R_n$ be the orthogonal projection onto the linear span of the set $\{e_k :k\leq n\}.$ Since for every $n\in\N$ the space $R_n \mathcal{B}(H) R_n$ is finite-dimensional, choose $D_n$ a countable dense subset that contains $R_n$. For $n<m\in \N$, we also require  that $D_n\subseteq R_nD_mR_n.$ 

Similar to the case of Boolean algebras, we define a forcing notion for a quasidiagonal \cstar-algebra whose  conditions represent  partial maps from a finite subset of $A$ to an ``initial segment'' in $\mathcal{B}(H)$, which in this case is a \textit{corner} $R_n\mathcal{B}(H)R_n$ for some $n\in\N$. 
Extensions of conditions are defined  as to yield better approximations, maps are defined on a bigger domain and take values on a larger corner in $\mathcal{B}(H)$. It is only on a sufficient part of the larger corner that we shall request that the new maps preserve the norm of elements  and all algebraic operations, modulo a  small error (which disappears once one passes to the Calkin algebra).

\begin{definition}\label{qddef}
Let $A$ be a unital, quasidiagonal \cstar-algebra  and define ${\mathbb{Q}\mathbb{D}}_A$ to be the set of all tuples
\[
p = (F_p,n_p,\epsilon_p,\psi_p)
\]
such that:
\begin{enumerate}
\item $F_p\Subset A$ is such that $1\in F_p$,
\item $n_p\in\N$,
\item $\epsilon_p\in {\Q}^+$,
\item $\psi_p:F_p\rightarrow D_{n_p}$ is a unital map such that 
$\lVert \psi_p(a)\rVert \le\lVert a\rVert $ \text{for all}  $a\in F_p
$. This map is \textit{not} required to be linear or self-adjoint.\label{qdmap}
\pushcounter
\end{enumerate}
For $p,q\in {\mathbb{Q}\mathbb{D}}_A$, we  write $p < q$ if the following hold:
\begin{enumerate}
\popcounter
\item $F_q\subseteq F_p$,
\item $n_q < n_p$,
\item $\epsilon_p < \epsilon_q$,
\item \label{item:qd} $\psi_p(a) R_{n_q} = R_{n_q} \psi_p(a) = \psi_q(a)$ for all $a\in F_q$,
\item  $\lVert \psi_p(a)(R_{n_p}-R_{n_q})\rVert > \lVert  a\rVert -\epsilon_q$ for all $a\in F_q$,
\item  for $a, b \in A$ and $\lambda, \mu \in \C$ define
\begin{align*}
\Delta^{p,+}_{a,b,\lambda,\mu} &:= \psi_p(\lambda a + \mu b) - \lambda \psi_p(a) - \mu \psi_p(b),
\\
\Delta^{p,*}_a 
&:= \psi_p( a^*) -  \psi_p(a)^*,
\\
\Delta^{p,\cdot}_{a,b} 
&:= \psi_p(a  b) - \psi_p(a) \psi_p(b).
\end{align*}
Then we require
\begin{enumerate}
\item $\lVert \Delta^{p,+}_{a,b,\lambda,\mu}(R_{n_p}-R_{n_q})\rVert < \epsilon_q - \epsilon_p$  if $a,b,\lambda a +\mu b\in F_q,$
\item $\lVert \Delta^{p,*}_a(R_{n_p}-R_{n_q})\rVert < \epsilon_q - \epsilon_p$ if $a,a^*\in F_q,$
\item $\lVert \Delta^{p,\cdot}_{a,b}(R_{n_p}-R_{n_q}) \rVert < \epsilon_q - \epsilon_p$ if $a,b,ab\in F_q.$
\end{enumerate}
\end{enumerate}
\end{definition}

Item \ref{item:qd} above displays the block-diagonal fashion of the extension of conditions and plays a crucial role in ascertaining that the relation $<$ is transitive. To demonstrate it, by considering multiplication as an example, for conditions $p < q < s$ in ${\mathbb{Q}\mathbb{D}}_A$
we have that
\begin{align*}
\lVert \Delta^{p,\cdot}_{a,b}(R_{n_p}-R_{n_s}) \rVert 
&\leq \lVert \Delta^{p,\cdot}_{a,b}(R_{n_p}-R_{n_q}) \rVert + \lVert \Delta^{p,\cdot}_{a,b}(R_{n_q}-R_{n_s}) \rVert 
\\
&< \epsilon_q - \epsilon_p + \lVert \Delta^{p,\cdot}_{a,b}(R_{n_q}-R_{n_s}) \rVert.
\end{align*}
Item \ref{item:qd} implies that
\[ \psi_p(c)(R_{n_q}-R_{n_s}) = \psi_q(c)(R_{n_q}-R_{n_s}) = (R_{n_q}-R_{n_s})\psi_q(c)(R_{n_q}-R_{n_s}),
\]
for all $c\in F_s$. Thus 
\begin{multline*}
\psi_p(a)\psi_p(b)(R_{n_q}-R_{n_s}) = \psi_p(a) (R_{n_q}-R_{n_s})\psi_q(b)(R_{n_q}-R_{n_s}) \\
= \psi_q(a)\psi_q(b)(R_{n_q}-R_{n_s}),
\end{multline*}
which in turn yields
\[
\lVert \Delta^{p,\cdot}_{a,b}(R_{n_q}-R_{n_s}) \rVert < \epsilon_s - \epsilon_q.
\]

Note that for any finite set $F\Subset A$ and $n\in\N$ there are only countably many maps $\psi:F\rightarrow D_n$ as in condition \ref{qdmap}. This, along with a standard uniformization argument and an application of the $\Delta$-System Lemma (Lemma \ref{delta}), reduces 
(similarly to the case of Boolean algebras) the problem of whether the poset $\mathbb{Q}\mathbb{D}_A$ is ccc to the following:

\begin{lemma}\label{qdcomp}
Let $p,q\in {\mathbb{Q}\mathbb{D}}_A$  be two conditions such that $n_p=n_q , \epsilon_p=\epsilon_q$ and the maps $\psi_p,\psi_q$ agree on $F_p\cap F_q$. Then, $p$ and $q$ are compatible.
\end{lemma}

To see this, for $\epsilon_s = \epsilon_p/8$ and $F_s=F_p\cup F_q$, let $m\in\N$ and  $\phi: F_s\rightarrow M_m(\C)$ be given as in Definition \ref{uqddef}. By setting $n_s=n_p +m$, identifying $M_m(\C)$ with the corner $(R_{n_s}-R_{n_p})\mathcal{B}(H)(R_{n_s}-R_{n_p})$ and approximating $\phi$ via the dense sets up to $\epsilon_s$, define a map $\psi_s$ which block-diagonally extends both $\psi_p$ and $\psi_q$ via this approximation of $\phi$. In this manner, the resulting condition $s = (F_s,n_s,\epsilon_s, \psi_s)\in {\mathbb{Q}\mathbb{D}}_A$ extends both $p$ and $q$.

The previously described argument also gives the basic idea of how to extend a given condition
(allowing also to enlarge the domain) by diagonally adjoining a finite-dimensional block in which, modulo a small error, all algebraic operations and the norm of all elements are preserved. This hints that a generic filter
induces (analogously to the case of Boolean algebras in the previous subsection; see also
Proposition \ref{genemb}) a map from $A$ into $\cQ(H)$
which is an isometric (and thus injective) $\ast$-homomorphism.

\section{The General Case}\label{sctn4}
In this section we proceed to define the forcing notion $\mathbb{E}_A$ and give the proof of Theorem~\ref{T1}.

\subsection{The Definition of the Poset} \label{subsctn4}
For what follows let $A$ be a simple, unital \cstar-algebra. We begin by fixing
an increasing
countable family of projections $\mathcal{P} \subseteq \mathcal{B}(H)$ 
converging strongly to the identity and
a countable dense subset $C$ 
of $\Bf(H)^{\le 1}_+$.
For  $R \in \mathcal{P}$
and $h \in C$ let $S_{R,h}$ be the orthogonal projection onto the span of $h^+[H]\cup R[H]$.
Fix a countable dense subset
\[
D_{R,h}\subseteq \set{S_{R,h} T h^+: T\in  \mathcal{B}(H)}
\]
 that contains $h^+$. We need the dense sets $D_{R,h}$ and $C$ to satisfy certain closure properties in order to carry out the arguments below. We
explicit these properties in detail here, but the reader can safely
ignore them for now and come back to them
when reading the proof of Proposition~\ref{prop:1}.

\begin{definition} \label{dense}
The countable sets $C$ and $D_{R,h}$ previously defined are required to have the following closure properties. 
\begin{enumerate}
\item \label{dense1}
For all $c_1, \dots, c_k \in C$ and $R \in \mathcal{P}$,
 the intersection of $C$ with the set  (recall that  $h\gg c$ stands for  $hc=c$)
\[
\set{h \in \Bf(H)_+^{\le 1} : h \gg c_1, \dots,  h \gg c_k,  h \ge R}
\]
is dense in the latter. 
\item \label{dense2}
Given $R \in \mathcal{P}$ and $h,k \in C$, the intersection 
of $D_{R,h}$ with the set
\[
 \set{T \in S_{R,h} \mathcal{B}(H)h^+ : Tk^-[H] \subseteq h^-[H],  Th^-[H] \subseteq h^+[H]}
\]
is dense in the latter. 
\item \label{dense3}
Given $R, R' \in \mathcal{P}$, 
 $h_1, h_2 ,k \in C$, and $T' \in D_{R',h_2}$, 
  the intersection of $D_{R,h_1}$
 with the set
\begin{align*}
\set{T \in S_{R,h_1}\mathcal{B}(H) h_1^+ :& Th_1^+ = T' ,  h_2^-T =
h_2^-T' ,\\ &T k^-[H] \subseteq h_1^-[H], Th_1^-[H] \subseteq h_1^+[H]}
\end{align*}
is dense in the latter.
\end{enumerate}
\end{definition}

It is straightforward to build countable dense sets with such properties by 
countable iteration.\footnote{A logician 
can use a large enough countable 
elementary submodel of a sufficiently large hereditary set containing all the relevant objects as a parameter 
to outright define these sets.}
This idea appears in \cite{Wof:Set}, where ccc forcing was  used to study the poset of projections in the 
Calkin algebra.

Before proceeding to the definition of the poset, we pause to give some insight  and justify the considerably higher complexity it possesses when
compared with the abelian or quasidiagonal  case.
The rough idea is, again,
to define a poset where each condition represents a partial map from a finite subset
of $A$ into some finite-dimensional corner of $\cB(H)$ and where the ordering guarantees
that stronger conditions behave like $*$-homomorphisms on larger and larger subspaces
of $H$ up to an error which tends to zero. The countable, dense sets $D_{R,h}$ considered in the beginning of this section serve as the codomains of these partial maps and, as a result, for any finite subset of $A$ there are only countable many possible maps into any given corner.
The main difference with the quasidiagonal case
is that we cannot expect conditions to look like block-diagonal matrices anymore.
This has troublesome consequences,
mostly caused by the multiplication (and to a minor extent 
by the adjoint operation). The main issue is that, given $p< q$, one cannot expect that a property similar to condition \ref{item:qd} of Definition \ref{qddef}, that is
\[
R_{n_q}\psi_p(a)(1-R_{n_q}) =  (1-R_{n_q})\psi_p(a) R_{n_q} =  0,
\]
can hold in general. As a first consequence (and with the comments succeeding Definition \ref{qddef} in mind), even defining a partial order that is transitive proves to be non-trivial. An even bigger
issue that comes up  is the extension of a condition 
to a stronger one with larger domain.
While in the quasidiagonal case it is sufficient to add a finite-dimensional block
with some prescribed properties, completely ignoring how $\psi_p$ is defined, in the general case
one has to explicitly require for $\psi_p$ to allow at least one extension in order to avoid~$\mathbb{E}_A$ having atomic conditions\footnote{Given a poset $(P,<)$, $p \in P$ is \emph{atomic} if $q \le p$ implies $q=p$.}. To this end, the poset $\mathbb{E}_A$ is defined as follows:

\begin{definition} \label{def:1}
Let $\mathbb{E}_A$ be the set of the tuples
\[
p = (F_p, \epsilon_p, h_p, R_p, \psi_p)
\]
where
\begin{enumerate}
\item $F_p \Subset A$, $1 \in F_p$ and if $a \in F_p$ then $a^* \in F_p$, 
\item $ \epsilon_p \in \mathbb{Q}^+$,
\item $h_p \in C$,
\item $R_p \in \mathcal{P}$,
\item $\psi_p: F_p \to D_{R_p,h_p}$ and there exist a faithful, essential, unital $*$-homomorphism
$\Phi_p : C^\ast(F_p) \to \mathcal{B}(H)$ and a projection $k_p \le h_p^-$ such that for all
$a \in F_p$ \label{promise}
\begin{enumerate}
\item $k_p = k^-$ for some $k \in C$,
\item $\psi_p(1) = h^+_p$,
\item \label{itemc} $\lVert (\psi_p(a) - \Phi_p(a))(h_p^+ - k_p) \rVert < \frac{\epsilon_p}{3M_p}$, where
\begin{align*}
L(F_p) = \max \set{\lvert \lambda \rvert : \lambda \in \C\ \text{and} \ \exists &\mu \in \C, \  \exists a,b \in F_p\\
 &\text{s.t.} \  a\neq 0 \ \text{and}  \      \lambda a + \mu b \in F_p}
\end{align*}
and
\[
M_p = \max \set{3\lVert a \rVert, 3\lVert \psi_p(a) \rVert, L(F_p) : a \in F_p},
\]
\item $\lVert \psi_p(a) + \Phi_p(a)(1- h_p^+) \rVert < \frac{3}{2} \lVert a \rVert$,
\item $ \psi_p(a) k_p[H] \subseteq h_p^-[H] $ and  $ \psi_p(a) h_p^-[H] \subseteq h_p^+[H]$,
\label{relf}
\item \label{relg} $ \Phi_p(a) k_p[H] \subseteq h_p^-[H]$ and $ \Phi_p(a) h_p^-[H] \subseteq h_p^+[H]$.
\end{enumerate}
\pushcounter
\end{enumerate}

Such pair $(k_p, \Phi_p)$  will henceforth be referred to as a \emph{promise} for the condition $p$.

Given $p, q \in \mathbb{E}_A$, we say that \textit{$p$ is stronger than $q$} and write $p < q$ if and only if
\begin{enumerate}
\popcounter
\item $F_p \supseteq F_q$, \label{rel:1}
\item $\epsilon_p < \epsilon_q$ \label{rel:2}
\item $h_p \gg h_q$, \label{rel:3}
\item $R_p\geq R_q$, \label{rel:4}
\item $\psi_p(a) h_q^+ = \psi_q(a)$ for all $a \in F_q$, \label{rel:5}
\item $h_q^- \psi_p(a) = h_q^- \psi_q(a)$ for all $a \in F_q$, \label{rel:6}
\item  \label{rel:7}
\begin{enumerate}
\item $\lVert \Delta^{p,+}_{a,b,\lambda,\mu}(h_p^- - h_q^-) \rVert
< \epsilon_q - \epsilon_p$ for $a,b,\lambda a + \mu b \in F_q$, \label{rel:a}
\item $\lVert \Delta^{p,*}_a(h_p^- - h_q^-) \rVert < \epsilon_q - \epsilon_p$
for $a \in F_q$, \label{rel:b}
\item $\lVert \Delta^{p,\cdot}_{a,b}(h_p^- - h_q^-) \rVert < \epsilon_q - \epsilon_p$
for $a,b,ab \in F_q$, \label{rel:c}
\end{enumerate}
where the quantities $\Delta^{p,+}_{a,b,\lambda,\mu}$, $\Delta^{p,*}_a$  and $\Delta^{p,\cdot}_{a,b}$ are defined as in Definition  \ref{qddef}.
\end{enumerate}

\end{definition}
Item \ref{relf} above is an example of how the problem of transitivity is addressed and this becomes clear in Claim \ref{claim*} of the next proposition. The promise in item \ref{promise} is witnessing that there is at least one way to extend $p$
(via $\Phi_p$) to conditions with arbitrarily large (finite-dimensional) domain. We will see later
(see Propositions \ref{prop:1}, \ref{lemma:3} and \ref{propK}) how Theorem \ref{lemma:2} and Corollary \ref{voic} imply that the choice of a specific
$\Phi_p$ is not a real constraint on how extensions of $p$ are going to look like.

\begin{proposition}
The relation $<$ defined on $\mathbb{E}_A$ is transitive.
\end{proposition}
\begin{proof}
Let $p,q,s \in \mathbb{E}_A$ be such that $p < q < s$.
It is straightforward to check that conditions \ref{rel:1}--\ref{rel:4} hold between $p$ and $s$.
Items \ref{rel:5} and \ref{rel:6} follow since $ h_q \gg h_s$ implies $ h_q^- \ge h_s^+$.
We recall that for two projections $p, q$ the relation $p \le q$ is equivalent to $pq = qp = p$.
We divide the proof of condition \ref{rel:7} in three claims, one for each item.
\begin{claim}
If $a,b,\lambda a + \mu b \in F_s$ then
$\lVert \Delta^{p,+}_{a,b,\lambda,\mu}(h_p^- - h_s^-) \rVert < \epsilon_s- \epsilon_p$.
\end{claim}
\begin{proof}
We have
\[
\lVert \Delta^{p,+}_{a,b,\lambda,\mu}(h_p^- - h_s^-) \rVert
\le
\lVert \Delta^{p,+}_{a,b,\lambda,\mu}(h_p^- - h_q^-) \rVert
+ \lVert \Delta^{p,+}_{a,b,\lambda,\mu}(h_q^- - h_s^-) \rVert.
\]
Since $p < q < s$, we know that $\psi_p(c)h_q^+ = \psi_q(c)$ for all $c \in F_q$ (item \ref{rel:5})
and thus $\lVert \Delta^{p,+}_{a,b,\lambda,\mu}(h_q^- - h_s^-) \rVert =
\lVert \Delta^{q,+}_{a,b,\lambda,\mu}(h_q^- - h_s^-) \rVert $. Hence
we can conclude
\[
\lVert \Delta^{p,+}_{a,b,\lambda,\mu}(h_p^- - h_q^-) \rVert +
\lVert \Delta^{p,+}_{a,b,\lambda,\mu}(h_q^- - h_s^-) \rVert < \epsilon_q - \epsilon_p + \epsilon_s - \epsilon_q = \epsilon_s - \epsilon_p,
\]
as required.
\end{proof}

\begin{claim}
If $a \in F_s$ then
$\lVert \Delta^{p,*}_a(h_p^- - h_s^-) \rVert < \epsilon_s - \epsilon_p$.
\end{claim}
\begin{proof}
We have
\[
\lVert \Delta^{p,*}_a(h_p^- - h_s^-) \rVert \le \lVert \Delta^{p,*}_a(h_p^- - h_q^-) \rVert
+ \lVert \Delta^{p,*}_a(h_q^- - h_s^-) \rVert.
\]
Since $p < q < s$, for all $c \in F_q$
we have that $\psi_p(c)h_q^+ = \psi_q(c)$ and that  $h_q^- \psi_p(c)
= h_q^- \psi_q(c)$ (items \ref{rel:5} and \ref{rel:6}). The latter relation entails that $\psi_p(c)^* h_q^- = \psi_q(c)^* h_q^-$. Thus, we conclude 
\begin{align*}
\lVert \Delta^{p,*}_a(h_p^- - h_q^-) \rVert
+ \lVert \Delta^{p,*}_a(h_q^- - h_s^-) \rVert &= \lVert \Delta^{p,*}_a(h_p^- - h_q^-) \rVert
+ \lVert \Delta^{q,*}_a(h_q^- - h_s^-) \rVert \\
&< \epsilon_s - \epsilon_p, 
\end{align*}
as required. 
\end{proof}
\begin{claim} \label{claim*}
If $a,b,ab \in F_s$ then
$\lVert \Delta^{p,\cdot}_{a,b}(h_p^- - h_s^-) \rVert < \epsilon_s - \epsilon_p$.
\end{claim}
\begin{proof}
We have
\begin{align*}
\lVert \Delta^{p,\cdot}_{a,b}(h_p^- - h_s^-) \rVert  
&\le \lVert \Delta^{p,\cdot}_{a,b}(h_p^- - h_q^-) \rVert +
\lVert \Delta^{p,\cdot}_{a,b}(h_q^- - h_s^-) \rVert \\
& <
 \epsilon_q - \epsilon_p + \lVert \Delta^{p,\cdot}_{a,b}(h_q^- - h_s^-) \rVert.
\end{align*} 
Since $\psi_p(c) h_q^+ = \psi_q(c)$
for all $c \in F_q$ (item \ref{rel:5}) we get
\[
(\psi_p(a  b) - \psi_p(a) \psi_p(b))(h_q^- - h_s^-) = (\psi_q(a  b) - \psi_p(a) \psi_q(b))(h_q^- - h_s^-)
\]
and therefore $(\psi_p(a  b) - \psi_p(a) \psi_p(b))(h_q^- - h_s^-)$ is equal to
\[
  \Delta^{q,\cdot}_{a,b}(h_q^- - h_s^-) +
(\psi_q(a) - \psi_p(a))\psi_q(b)(h_q^- - h_s^-).
\]
The rightmost term is zero since
$ \psi_q(b) \xi \in h_q^+[H]$ for all $\xi \in h_q^-[H]$ (item \ref{relf}) and
$\psi_p(a) h_q = \psi_q(a) h_q$ (this follows from item \ref{rel:5}).
This ultimately leads to the thesis since
$\lVert \Delta^{q,\cdot}_{a,b}(h_q^- - h_s^-) \rVert  < \epsilon_s - \epsilon_q$.
\end{proof}
This completes the proof.
\end{proof}

\subsection{Density and the Countable Chain Condition}\label{subsctn5}
As in Definition \ref{def:1}, for $F\Subset A$, let 
\begin{align*}
L(F) = \max\{|\lambda |:\lambda\in\C \ \text{and} \ \exists\mu\in\C,&\exists a,b\in F \\ &\text{s.t. } a\neq 0 \text{ and } \lambda a+\mu b\in F\}
\end{align*}
and
\[
J(F) = \max \set{ \lVert a \rVert : a \in F}.
\]
For $p\in\mathbb{E}_A$, let 
\[
M_p = \max\{3\lVert a\rVert,3\lVert \psi_p(a)\rVert,L(F_p) : a\in F_p \}.
\]
For $F\Subset A$ and $p\in\mathbb{E}_A$ let
\[
M(p,F) = 3\max\{3M_p +1,L(F), 2J(F)+1\}.
\]
Finally, for $p\in\mathbb{E}_A$ and a fixed promise $(k_p,\Phi_p)$ for the condition $p$, define the  constants
\[
N(p,\Phi_p) = \max\{\lVert (\psi_p(a)-\Phi_p(a))(h_p^+-k_p)\rVert : a\in F_p\}
\]
and
\[
D(p,\Phi_p) = \min\{3\lVert a\rVert/2 - \lVert \psi_p(a)+\Phi_p(a)(1-h_p^+)\rVert : a\in F_p\}.
\]
The main density result reads as follows:

\begin{proposition} \label{prop:1}
Given $F \Subset A$, $\epsilon \in \mathbb{Q}^+$,
$h \in C$ and $R \in \mathcal{P}$, the set
\[
\mathcal{D}_{F,\epsilon,h,R} = \set{p \in \mathbb{E}_A : F_p \supseteq F,
\epsilon_p \le \epsilon, h_p \gg h, R_p \ge R}
\]
is open dense in $\mathbb{E}_A$.
\end{proposition}
\begin{proof}
Clearly $\mathcal{D}_{F,\epsilon,h,R}$ is open.
Fix a condition
$q = (F_q, \epsilon_q, h_q, R_q, \psi_q)$ and
let $(k_q,\Phi_q)$ be a promise for the condition $q$.
By item \ref{itemc} of Definition \ref{def:1} there is a $\delta$ such that
\[
N(q,\Phi_q) < \delta
< \frac{\epsilon_q}{3M_q}
\]
Fix moreover a \emph{small enough} $\gamma$, more precisely
\[
\gamma \le \min \set{\epsilon, \epsilon_q - 3M_q \delta, D(q,\Phi_q)}.
\]
Let $F_p = F_q \cup F \cup F^*$. Applying Theorem \ref{lemma:2}, let $\Phi$
be a faithful, essential, unital representation of $C^\ast(F_p)$ such that
\[
\lVert \Phi_{\restriction F_q} - \Phi_{q \restriction F_q} \rVert < \frac{\gamma}{36M}
\]
with $M = M(q,F_p)$.
Consider, by condition \ref{dense1} of Definition \ref{dense}, an operator $k \in C$ such that $k \gg h, k \gg h_q, k \gg R_q$ and denote $k^-$ by $k_p$.
Let $T$ be the finite-rank projection onto the space spanned by the  set $\set{\Phi(a)k[H]: a \in F_p}$. By item \ref{dense1} of Definition \ref{dense}, since $T \gg k$,
we can choose $l \in C$ such that $l \gg k$ and $l \approx_{\frac{\gamma}{18M}} T$. Moreover,
by Lemma \ref{lemma:1}, picking $l$ closer to $T$ if needed, there is a unitary
$u\in\mathcal{U}(H)$ such that 
\begin{enumerate}
\item $u$ is a compact perturbation of the identity,
\item $uT[H] \subseteq l[H]$,
\item $u$ is the identity on $k_p[H]$ (since $l \gg k_p$),
\item $\lVert (\Ad u (\Phi(a))  - \Phi(a))k_p \rVert < \frac{\gamma}{36M}$ for all $a \in F_p$.\pushcounter
\end{enumerate}
This entails that $\Phi' = \Ad u \circ \Phi $
is such that $\Phi'(a)k_p[H] \subseteq l[H]$ and
\[
\lVert (\Phi'(a) - \Phi_q(a) ) k_p \rVert < \frac{\gamma}{18M}
\]
for all $a \in F_q$. Let $Q$ be the finite-rank projection onto the space spanned by the set $\set{\Phi'(a)l[H] : a \in F_p}$ and let $K$ be the finite-rank
operator equal to the identity on $l[H]$, equal to $\frac{1}{2}\Id$ on $Q(H) \cap {l[H]}^\perp$
(remember that $Q \ge l^+$ since $1 \in F_p$) and equal to zero on ${Q[H]}^\perp$. By item \ref{dense1} of Definition \ref{dense} there is $h_p \in C$ such that $h_p \gg l$ and
$h_p \approx_{\frac{\gamma}{15M}} K$. Moreover, by picking $h_p$ closer to $K$ if necessary, we
may assume that $\text{dim}(h_p Q[H]) = \text{dim}(Q[H])$ and that $h_p^- = l^+$. The first equality
can be obtained with the argument exposed at the beginning of the proof of Lemma \ref{lemma:1},
while the second is as follows: Suppose $\xi \in {l[H]}^\perp$ is
a norm one vector, then $\xi = \xi_1 + \xi_2$, where
$\xi_1$ and $\xi_2$ are orthogonal vectors of norm
smaller than 1 such that $K \xi_1 = \frac{1}{2} \xi_1$
and $K \xi_2 = 0$. Hence, if $h_p$ is close enough to $K$ it follows that $\lVert h_p \xi \rVert < 1$.
The equality $\text{dim}(h_p Q [H]) = \text{dim}(Q[H])$ allows us to find a unitary $v$ such that
\begin{enumerate}
\popcounter
\item $v$ is a compact perturbation of the identity,
\item $v$ sends $Q[H]$ in $h_p[H]$,
\item $v$ is the identity on $l[H]$.
\pushcounter
\end{enumerate}
The representation $\Phi_p = (\Ad v)\circ \Phi' $ is such that
\begin{enumerate}
\popcounter
\item $\Phi_p(a) k_p[H] \subseteq h_p^-[H]$ for all $a \in F_p$,
\item $\Phi_p(a) h_p^-[H] \subseteq h_p^+[H]$ for all $a \in F_p$,
\item \label{proof10} $\lVert (\Phi_p(a) - \Phi_q(a) ) k_p \rVert < \frac{\gamma}{18M}$ for all $a \in F_q$.
\pushcounter
\end{enumerate}
Let $R_p \in \mathcal{P}$ be such that $R_p \ge R, R_p\geq R_q $ and
\[
\lVert (1 - R_p) \Phi_p(a) h_p^+ \rVert < \frac{\gamma}{18M}
\]
for all $a \in F_p$. Consider now, given $a \in F_q$, the operator
\[
\phi(a) = \psi_q(a) + (1 - h_q^-)\Phi_p(a)(h_p^- - h_q^+) + (1 - h_q^-)R_p\Phi_p(a)(h_p^+ - h_p^-)
\]
and for $a \in F_p \setminus F_q$ the operator
\[
\phi(a) = \Phi_p(a)h_p^- + R_p \Phi_p(a)(h_p^+ - h_p^-).
\]
For all $a \in F_p$ we have $\phi(a) k_p[H] \subseteq h_p^-[H]$ and
$\phi(a) h_p^-[H] \subseteq h^+_p[H]$. Moreover, for $a \in F_q$ we also have
$\phi(a) h_q^+ = \psi_q(a)$ and $h_q^- \phi(a) = h_q^- \psi_q(a)$. Let
$\psi_p: F_p \to ~D_{R_p,h_p}$ be a function such that:
\begin{enumerate}
\popcounter
\item \label{psi1} $\psi_p(1) = h_p^+$,
\item \label{psi2} for all $a\in F_p$,  $\psi_p(a) \approx_{\frac{\gamma}{18M}} \phi(a)$ and we also require that
\begin{enumerate}
\item $\psi_p(a) k_p[H] \subseteq h_p^-[H]$ and $\psi_p(a) h_p^-[H] \subseteq h^+_p[H]$ for all $a\in F_p$,
\item $\psi_p(a) h_q^+ = \psi_q(a)$ and $h_q^- \psi_p(a) = h_q^- \psi_q(a)$ for all $a\in F_q$.
\end{enumerate}
\end{enumerate}
Such a function $\psi_p$ exists because of the requirements on $D_{R_p,h_p}$
we asked in items \ref{dense2} and \ref{dense3} of Definition \ref{dense}.
\begin{claim} \label{claim:1}
For all $a\in F_p$ we have $\lVert (\psi_p(a) - \Phi_p(a))(h_p^+ - k_p) \rVert < \frac{\gamma}{6M}$.
\end{claim}
\begin{proof}
The inequality is trivially true for $a = 1$. For $a \in F_p \setminus F_q$ we have
\[
\psi_p(a) (h_p^+ - k_p) \approx_{\frac{\gamma}{18M}}
\Phi_p(a)(h_p^- - k_p) + R_p \Phi_p(a)(h_p^+ - h_p^-) \approx_{\frac{\gamma}{18M}}
\Phi_p(a)(h_p^+ - k_p)
\]
since $h_q^+ (h_p^+ - k_p) = 0$, $(h_p^+ - k_p) \ge (h_p^- - h_q^+)$,
$(h_p^+ - k_p) \ge (h_p^+ - h_p^ -)$ and
where the last approximation is a consequence of how we defined $R_p$, in particular of
\[
\lVert (1 - R_p) \Phi_p(a) h_p^+ \rVert < \frac{\gamma}{18M}.
\]
Now let $a \in F_q \setminus \set{1}$. Similarly to the previous case we get
\[
\psi_p(a) (h_p^+ - k_p) \approx_{\frac{\gamma}{9M}}
(1-h_q^-) \Phi_p(a) (h_p^+ - k_p).
\]
By the definition of the promise (item \ref{relg} of Definition \ref{def:1}), we have that $(h_p^+ - h_q^+) \Phi_q(a) h_q^- = 0$. Remember that
by definition of $\Phi_p$ we have
\[
\lVert (\Phi_p(a) - \Phi_q(a) ) k_p \rVert < \frac{\gamma}{18M}.
\]
Use this inequality
and $k_p \ge h_q^-$
to infer that $(h_p^+ - h_q^+) \Phi_p(a) h_q^- \approx_{\frac{\gamma}{18M}} 0$. Since $F_q$ is self-adjoint, we also obtain that
\[
h_q^- \Phi_p(a) (h_p^+ - h_q^+) \approx_{\frac{\gamma}{18M}} 0.
\]
This allows us to conclude that $\psi_p(a)(h_p^+ - k_p)\approx_{\frac{\gamma}{6M}}\Phi_p(a)(h_p^+ - k_p).$
\end{proof}
\begin{claim}
For all $a\in F_p$ we have $\lVert \psi_p(a) + \Phi_p(a)(1 - h^+_p) \rVert < \frac{3}{2} \lVert a \rVert$.
\end{claim}
\begin{proof}
Let $a \in F_p \setminus F_q$. Then we have
\begin{align*}
\psi_p(a) + \Phi_p(a)(1 - h^+_p)  \approx_{\frac{\gamma}{18M}}
\Phi_p(a)h_p^- + R_p &\Phi_p(a)(h_p^+ - h_p^-) \\ + &\Phi_p(a)(1 - h^+_p) 
\approx_{\frac{\gamma}{18M}} \Phi_p(a).
\end{align*}
hence the thesis follows since $\lVert \Phi_p(a) \rVert \le \lVert a \rVert$ and we
can assume $\gamma \le
\lVert a \rVert$. Consider now $a \in F_q$. Since in the previous claim we showed that
\[
h_q^- \Phi_p(a) (h_p^+ - h_q^+) \approx_{\frac{\gamma}{18M}} 0,
\]
we have
\[
\psi_p(a) + \Phi_p(a)(1 - h^+_p)  \approx_{\frac{\gamma}{18M}} \phi(a) + \Phi_p(a)(1 - h^+_p)
\approx_{\frac{\gamma}{9M}} \psi_q(a) + \Phi_p(a) (1 - h^+_q).
\]
Recall that $\Phi_p = (\Ad w)\circ \Phi $, where $w$ is a unitary which behaves like the identity
on $k_p$ (hence on $h_q^+$ and $R_q$ as well), thus $w(1 - h_q^+) = (1 - h_q^+) w$ and $\psi_q(a) = \Ad w (\psi_q(a))$ for all $a \in F_q$.
Moreover $\Phi$ was defined so that
\[
\lVert \Phi_{\restriction F_q} - \Phi_{q \restriction F_q} \rVert < \frac{\gamma}{36M}.
\]
Therefore the following holds
\begin{align*}
\lVert \psi_q(a) + \Phi_p(a) (1 - h^+_q) \rVert &= \lVert \psi_q(a) + \Phi(a) (1 - h^+_q) \rVert \\
&\approx_{\frac{\gamma}{36M}} \lVert \psi_q(a) + \Phi_q(a) (1 - h^+_q) \rVert < \frac{3}{2}
\lVert a \rVert,
\end{align*}
which implies the thesis since $\gamma \le \lVert a \rVert$.
\end{proof}
This finally entails that, letting $\epsilon_p = \frac{\gamma}{6}$,
\[
p = (F_p, \epsilon_p, h_p, R_p, \psi_p)
\]
is an element of $\mathcal{D}_{F, \epsilon, h, R}$. It is in fact straightforward to check
that  if $\gamma$ is small enough, then $M_p \le M = M(q, F_p)$.
We are left with checking that $p < q$. The conditions \ref{rel:1}--\ref{rel:5} in Definition \ref{def:1} follow
from the definition of $p$.

\begin{claim} \label{claim:2}
For all  $a,b,\lambda a + \mu b \in F_q$ we have that
$\lVert (\Delta^{p,+}_{a,b,\lambda,\mu})(h_p^- - h_q^-) \rVert
< \epsilon_q - \epsilon_p$. 
\end{claim}

\begin{proof}
Given $c \in F_q$ we have, by definition of $\delta$ (see the beginning of the proof),
$\lVert (\psi_q(c) - \Phi_q(c))(h_q^+ - k_q) \rVert < \delta$, and the same is true if
we replace $(h_q^+ - k_q)$ with $(h_p^- - h_q^-)$, since $(h_q^+ - k_q) \ge
(h_p^- - h_q^-)$. Moreover, by definition of $\Phi_p$,
 $\lVert (\Phi_p(c) - \Phi_q(c) ) k_p \rVert < \frac{\gamma}{18M}$ holds. This, along
 with the fact that $F_q$ is self-adjoint, $\Phi_q(c) h_q^-[H] \subseteq h_q^+[H]$ (item
 \ref{relg} of Definition \ref{def:1}) and $k_p \ge h_q^+$,
entails that $\lVert h_q^- \Phi_p(c) (h_p^+ - k_p) \rVert < \frac{\gamma}{18M}$.
Therefore 
\[
(\Delta^{p,+}_{a,b,\lambda,\mu})(h_p^- - h_q^-) \approx_{\frac{\gamma}{6}}
(\phi(\lambda a + \mu b) - \lambda \phi(a) - \mu \phi(b)) (h_p^- - h_q^-)
\approx_{3M_q \delta + \frac{\gamma}{3}} 0
\]
as required. 
\end{proof}
\begin{claim} \label{claim:3}
For all  $a \in F_q$
we have $\lVert (\Delta^{p,*}_a)(h_p^- - h_q^-) \rVert < \epsilon_q - \epsilon_p$.
\end{claim}

\begin{proof}
Using approximations analogous to previous claim, we have that
\begin{align*}
 (\Delta^{p,*}_a)(h_p^- - h_q^-)  &\approx_{\frac{\gamma}{9}}
 (\phi(a^*) - \phi(a)^*)(h_p^- - h_q^-)\\
 &  \approx_{\delta + \frac{\gamma}{9}}
 (\Phi_p(a^*) - \psi_q(a)^* -(h_p^- - h_q^+)\Phi_p(a^*)(1 - h_q^-) \\
 &\qquad - (h_p^+ - h_p^-)
\Phi_p(a^*) R_p(1 - h_q^-))(h_p^- - h_q^-). 
\end{align*} 
Since $F_p$ is self-adjoint and by definition of $R_p$
\[
\lVert  h_p^+ \Phi_p(c) (1 - R_p)\rVert < \frac{\gamma}{18M}
\]
for all $c \in F_q$, thus $(h_p^+ - h_p^-) \Phi_p(a^*) R_p(1 - h_q^-) \approx_{\frac{\gamma}{18M}}
(h_p^+ - h_p^-) \Phi_p(a^*) (1 - h_q^-)$.
Hence we obtain
\[
 (\Delta^{p,*}_a)(h_p^- - h_q^-)  \approx_{\delta + 5\frac{\gamma}{18}}
 (\Phi_p(a^*) - \psi_q(a)^* - (h_p^+ - h_q^+) \Phi_p(a^*) (1 - h_q^-))(h_p^- - h_q^-).
\]
Furthermore we have
\begin{align*}
\psi_q(a)^* (h_p^- - h_q^-) = ((h_p^- - h_q^-) \psi_q(a))^* &=  ((h_p^- - h_q^-) \psi_q(a)h_q^+)^* \\
&= ((h_p^- - h_q^-) \psi_q(a)(h_q^+ - k_q))^*,
\end{align*}
where the last equality is a consequence of $\psi_q(c)k_q H \subseteq h_q^- H$ for all $c \in F_q$ (item \ref{relf} of Definition \ref{def:1}). Since
\[
\lVert (\psi_q(c) - \Phi_q(c))(h_q^+ - k_q) \rVert < \delta , \
\lVert (\Phi_p(c) - \Phi_q(c) ) k_p \rVert < \frac{\gamma}{18M},
\]
we get that
\[
 (\Delta^{p,*}_a)(h_p^- - h_q^-)  \approx_{2\delta + \frac{\gamma}{3}}
 \Phi_p(a^*)(h_p^- - h_q^-) - (h_p^+ - k_q)\Phi_p(a^*)(h_p^- - h_q^-).
\]
Moreover, by how we defined $\Phi_p$ we have
\[
\Phi_p(a^*)(h_p^- - h_q^-) =
h_p^+ \Phi_p(a^*)(h_p^- - h_q^-)
\]
and
\[
 (1 - h_q^-) \Phi_p(c) k_q   \approx_{\frac{\gamma}{18M}} (1 - h_q^-) \Phi_q(c) k_q = 0
\]
for all $c \in F_q$. This last approximation entails, since $F_q$ is self-adjoint, that
\[
\lVert k_q \Phi_p(c) (1 - h_q^-) \rVert < \frac{\gamma}{18M}
\]
for all $c \in F_q$.
\end{proof}
\begin{claim} \label{claim:4}
For all  $a,b,ab \in F_q$ we have
$\lVert (\Delta^{p,\cdot}_{a,b})(h_p^- - h_q^-) \rVert < \epsilon_q - \epsilon_p$. 
\end{claim}
\begin{proof}
Similarly to the previous claims, we have the following approximations
\begin{align*} 
(\Delta^{p,\cdot}_{a,b})(h_p^- - h_q^-) &\approx_{\frac{\gamma}{6}}
(\phi(ab) - \phi(a)\phi(b))(h_p^- - h_q^-) \\
& \approx_{2 M_q\delta + \frac{2\gamma}{9}}
\lVert (\Phi_p(ab) - \phi(a)\Phi_p(b))(h_p^- - h_q^-) \rVert.
\end{align*} 
As noted in the previous claim, for all $c \in F_q$ we have
\[
\lVert k_q \Phi_p(c) (1 - h_q^-) \rVert < \frac{\gamma}{18M},
\]
hence the same is true with $(h_p^- - h_q^-)$ in place of $(1 - h_q^-)$. Thus
\begin{align*}
\phi(a)\Phi_p(b)(h_p^- - h_q^-) &\approx_{\frac{\gamma}{18M}}
\phi(a)(1- k_q)\Phi_p(b)(h_p^- - h_q^-) \\
&\approx_{M_q\delta + \frac{\gamma}{6}}
\Phi_p(a)(1- k_q)\Phi_p(b)(h_p^- - h_q^-) \\
&\approx_{\frac{\gamma}{18M}}
\Phi_p(a)\Phi_p(b)(h_p^- - h_q^-), 
\end{align*}
as required. 
\end{proof}
This completes the proof.
\end{proof}
Let $B$ be the $(\Q+i\Q) \text{-} \ast$-algebra generated by a dense subset of $A$ with cardinality equal to the density character of $A$. We define the family $\mathcal{D}$ as follows ($C$ and $\cP$ 
were defined at the beginning of \S\ref{sctn4}):
\[
\mathcal{D} = \set{\mathcal{D}_{F,\epsilon,h,R}: F \Subset B, \epsilon \in \mathbb{Q}^+,
 h \in C, R \in \mathcal{P}}.
\]
\begin{proposition}\label{genemb}
Suppose there exists a $\mathcal{D}$-generic filter $G$ for $\mathbb{E}_A$. Then there
exists a unital embedding  of $A$ into the Calkin algebra.
\end{proposition}
\begin{proof}
Let $G$ be a $\mathcal{D}$-generic filter and fix $a \in B$. The net
$\set{\psi_p(a)}_{\set{p \in G : a \in F_p}}$ (indexed according to $(G, >)$, which is directed
since $G$ is a filter) is strongly convergent in $\mathcal{B}(H)$. Indeed, by Proposition \ref{prop:1} let
\[
p= p_0 > p_1 > \dots > p_n > \dots
\]
be an infinite decreasing sequence of elements of $G$ satisfying that $a \in F_p$,
$\epsilon_{p_n}<1/n$ and such that the sequence $\set{h_{p_n}}_{n \in \N}$ is an approximate unit
for $\mathcal{K}(H)$ (which is possible by density of $C$ and by genericity of $G$). The sequence
$\set{\psi_{p_n}(a)}_{n \in \N}$ is strongly convergent to an operator in $\mathcal{B}(H)$
(since $\lVert \psi_{p_n}(a) \rVert < 3 \lVert a \rVert /2$) which we denote by $\Psi(a)$.
In order to show that the whole net $\set{\psi_p(a)}_{\set{p \in G : a \in F_p}}$ strongly converges
to $\Psi(a)$, let $\xi_1, \dots, \xi_k$ be norm one vectors belonging to $h_{p_n}[H]$ for some
$n \in \N$. Then, for all $q \in G$ such that $q < p_n$ we have
\[
\psi_q(a) \xi_j = \psi_q(a) h_{p_n}^+ \xi_j = \psi_{p_n}(a) \xi_j
\]
for all $j \le k$. Since $\epsilon_n \to 0$ as $n \to \infty$, and $\set{h_{p_n}[H] : n \in \N}$
is dense in $H$ (by genericity, $(h_{p_n})_{n \in \N}$ is an approximate unit of $\mathcal{K}(H)$), it follows that
the net strongly converges to $\Psi(a)$ on $H$. Let $\Phi_G = \pi \circ \Psi$.
\begin{claim}
The map $\Phi_G:B \to \mathcal{Q}(H)$ defined above is a unital, bounded  $\ast$-homomorphism of $(\Q+i\Q)$-algebras.
\end{claim}
\begin{proof}
For $a, b \in B$, we will prove that $\Psi(ab) - \Psi(a)\Psi(b)$ is compact. Let $\epsilon > 0$
and pick $p \in G$ such that $a,b, ab \in F_p$ and $\epsilon_p < \epsilon$. We claim
that
\[
\lVert (\Psi(ab) - \Psi(a)\Psi(b))(1 - h_p^-) \rVert < \epsilon.
\]
Suppose this fails, and let $\xi \in (1 - h_p^-)H$ be a norm one vector such that
\[
\lVert (\Psi(ab) - \Psi(a)\Psi(b))\xi \rVert > \epsilon.
\]
By genericity of $G$ we can find $q \in G$ such that $q < p$ and
\[
\lVert (\Psi(ab) - \Psi(a)\Psi(b))\eta \rVert > \epsilon
\]
where $\eta = h_q \xi$. Now let $s < q $ in $G$ such that
$\Psi(b) \eta$ is close enough to $h_s \Psi(b) \eta$ to obtain
\[
\lVert (\psi_s(ab) - \psi_s(a)\psi_s(b))\eta \rVert > \epsilon.
\]
But this is a contradiction since $s < p$ implies
\[
\lVert (\psi_s(ab) - \psi_s(a)\psi_s(b))(h_s^- - h_p^-) \rVert < \epsilon_p < \epsilon.
\]
Similarly it can be checked that $\Phi_G$ is $(\Q+i\Q)$-linear and self-adjoint. Moreover, $\Phi_G$ is bounded since $\Psi$ is. The claim follows since $\Psi$ maps the unit of $A$ to the identity on~$H$.
\end{proof}
Extending $\Phi_G$ to the complex linear span of $B$, we obtain a unital, bounded $\ast$-homomorphism  into the Calkin algebra. This is a dense (complex) $\ast$-subalgebra of $A$, hence we can uniquely extend to obtain a unital $\ast$-homomorphism from $A$ into $\mathcal{Q}(H)$, which is injective, since $A$ is simple.
\end{proof}

Note that the fact that $\Phi_G$ above is bounded is crucial in allowing one to extend it and obtain a $*$-homomorphism defined on all of the algebra $A$. To see how this can fail, the identity map on the (algebraic) group algebra of any non-amenable discrete group cannot be extended to a $*$-homomorphism from the reduced group \cstar-algebra to the universal one (see \cite[Theorem~2.6.8]{brownozawa}). 


With the only part of Theorem \ref{T1} remaining unproven being the fact that the poset  is ccc, we begin with the following lemma yielding sufficient conditions for the compatibility of  elements of $\mathbb{E}_A$.

\begin{lemma} \label{lemma:3}
Suppose that  $p, q \in \mathbb{E}_A$ satisfy the following conditions.   
\begin{enumerate}
\item $h_p = h_q$ and $R_p = R_q$. 
\item $\psi_p(a) = \psi_q(a)$ for all $a \in F_p \cap F_q$. 
\item There exist two unital $*$-homomorphisms 
$\Phi_p:C^\ast (F_p)\rightarrow\mathcal{B}(H)$ and  $\Phi_q:~C^\ast (F_q)\rightarrow\mathcal{B}(H)$ which are faithful and essential,  
and a projection $k$ satisfying the following:
\begin{enumerate}
\item The pairs $(k,\Phi_p)$ and $(k,\Phi_q)$ are promises for  $p$ and $q$, respectively. 
\item \label{approx} There are constants $\delta_p$ and $\delta_q$ such that  
$N(p,\Phi_p)<\delta_p<\frac{{\epsilon}_p}{3M_p}$ and 
 $N(q,\Phi_q)<\delta_q<\frac{{\epsilon}_q}{3M_q}$, and 
if
\[
\gamma\leq\min\{{\epsilon}_p-3M_p{\delta_p}, D(p,\Phi_p), 
{\epsilon}_q-3M_q{\delta_q}, D(q,\Phi_q)\}
\]
and 
\[
M=\max \set{M(p, F_p \cup F_q), M(q,F_p \cup F_q)},
\] 
then  every $a\in F_p\cap F_q$ satisfies
$\lVert \Phi_p(a) - \Phi_q(a) \rVert < \frac{\gamma}{18M}$. 
\item There is a trivial embedding $\Theta: C^\ast(F_p \cup F_q) \to \mathcal{Q}(H)$   such that
$\pi \circ \Phi_p = \Theta_{\restriction C^\ast(F_p)}$ and
$\pi \circ \Phi_q = \Theta_{\restriction C^\ast(F_q)}$.
\end{enumerate}
\end{enumerate}
Then $p$ and $q$ are compatible.
\end{lemma}
\begin{proof}
Write  $h$ for $h_p$ and $R$ for  $R_q$. Let $\Phi$ be a faithful, essential, unital
representation that lifts $\Theta$ to $\mathcal{B}(H)$. 
Since $\Phi_p$ and $\Phi_{\restriction F_p}$ 
 agree modulo the compacts, and 
$\Phi_q$ and $\Phi_{\restriction F_q}$ agree modulo the compacts, 
there exists (by condition \ref{dense1} of Definition \ref{dense}) $k \in C$ such that 
 $k \gg h$, $k\gg R$, and in addition  
the following holds: 
For all $a\in F_p$ we have 
\[
\lVert (\Phi_p(a) - \Phi(a))(1 - k^-) \rVert < \frac{\gamma}{36M} , 
\]
and for all $a\in F_q$ we have 
\[
\lVert (\Phi_q(a) - \Phi(a))(1 - k^-) \rVert < \frac{\gamma}{36M}. 
\]
We shall denote $k^-$ by $k_s$.
Arguing as in the first part of the proof of Proposition \ref{prop:1} we can find $h_s \gg k_s$
(i.e. $h_s^- \ge k_s$) in $C$ and a unitary $w$ such that
\begin{enumerate}
\item $w$ is a compact perturbation of the identity,
\item $w k_s = k_s w = k_s$,
\pushcounter
\end{enumerate}
and by letting $\Phi'_p = (\Ad w)\circ\Phi_p$, $\Phi'_q = (\Ad w)\circ\Phi_q$ and 
$\Phi' = (\Ad w)\circ\Phi$, we also have that
\begin{enumerate}
\popcounter
\item $\lVert (\Phi'_p(a) - \Phi_p(a))k_s \rVert < \frac{\gamma}{36M}$
for all $a \in F_p$,
\item $\lVert (\Phi'_q(a)- \Phi_q(a))k_s \rVert < \frac{\gamma}{36M}$
for all $a \in F_q$,
\item $\lVert (\Phi'(a) - \Phi(a))k_s \rVert < \frac{\gamma}{36M}$
for all $a \in F_p\cup F_q$,
\item $\Phi'_p(a)k_s[H]\subseteq h_s^-[H]$ and $\Phi'_p(a)h_s^-[H]\subseteq h_s^+[H]$ for all $a\in F_p$,
\item $\Phi'_q(a)k_s[H]\subseteq h_s^-[H]$ and $\Phi'_q(a)h_s^-[H]\subseteq h_s^+[H]$ for all $a\in F_q$,
\item $\Phi'(a)k_s[H]\subseteq h_s^-[H]$ and $\Phi'(a)h_s^-[H]\subseteq h_s^+[H]$ for all $a\in F_p\cup F_q$.
\end{enumerate}
Let $R_s \in \mathcal{P}$ be such that $R_s \ge R$ and for all $a\in F_p$ and all $b\in F_q$ we have 
\begin{align*}
\lVert (1 - R_s)\Phi'_p(a)h_s^+ \rVert &< \frac{\gamma}{18M}, \\
\lVert (1 - R_s)\Phi'_q(b)h_s^+ \rVert &< \frac{\gamma}{18M}. 
\end{align*}
Given $a \in F_p$, consider the operator
\[
\phi(a) = \psi_p(a) + (1 - h^-)\Phi'_p(a)(h_s^- - h^+) + (1 - h^-)R_s\Phi'_p(a)(h_s^+ - h_s^-)
\]
and
for $a \in F_q \setminus F_p$
\[
\phi(a) = \psi_q(a) + (1 - h^-)\Phi'_q(a)(h_s^- - h^+) + (1 - h^-)R_s\Phi'_q(a)(h_s^+ - h_s^-).
\]
Define now the function $\psi_s : F_p \cup F_q \to D_{R_s,h_s}$ as an approximation
of $\phi$ in the same way it was done in the proof of Proposition \ref{prop:1}.
Suitably adapting the arguments in such proof to the present situation
it is possible to show that
\[
s = (F_p \cup F_q, \gamma/6, h_s, R_s, \psi_s)
\]
is an element of $\mathbb{E}_A$ with promise
$(k_s,\Phi')$. We follow the proof of Claim \ref{claim:1} in order to check that
the quantity $\lVert (\psi_s(a) - \Phi'(a))(h_s^+ - k_s) \rVert$ is small enough for $a \in F_p \cup F_q$, using in addition that for all  $a \in F_p$ 
\[
\lVert (\Phi_p(a) - \Phi(a))(1 - k_s) \rVert < \frac{\gamma}{36M} 
\]
and that for all $a\in F_q$ 
\[
\lVert (\Phi_q(a) - \Phi(a))(1 - k_s) \rVert < \frac{\gamma}{36M}. 
\]
This entails the same inequality between $\Phi'_p$ and $\Phi'$ (and between
$\Phi'_q$ and $\Phi'$) since the unitary $w$ fixes $k_s$. The proofs of
$s < p$ and $s < q$ go along the lines of those in Claim \ref{claim:2}, \ref{claim:3}
and \ref{claim:4}, keeping the following caveat in mind: It might happen, for instance,
that $p$ and $q$ are such that $a \in F_p \cap F_q$ and $b, ab \in F_q \setminus F_p$. In
this case $\Delta^{q, \cdot}_{a,b} (h_s^- - h_q^-)$ can be approximated (following the proof
of Claim \ref{claim:4}) as $(\Phi_q(ab) - \Phi_p(a) \Phi_q(b))(h_s^- - h_q^-)$. This is where
the condition $\Phi_p(a) \approx_{\frac{\gamma}{18M}} \Phi_q(a)$, required in
item \ref{approx} of the statement of the present lemma, plays a key role, showing that
the latter term is
close to zero. The same argument applies for the analogous situations where
$\Phi_p$ and $\Phi_q$ appear in the same formulas for the addition and the adjoint operation.
\end{proof}

Property K is a strengthening of the countable chain condition (see section \S \ref{sctn2.2}). 

\begin{proposition} \label{propK}
The poset $\mathbb{E}_A$ has property K and hence satisfies  the countable chain condition.
\end{proposition}
\begin{proof}
Let $\set{p_\alpha : \alpha < \aleph_1}$ be a set of conditions\footnote{We
suppress the notation and denote $F_{p_\alpha}$ by $F_\alpha$,
$\epsilon_{p_\alpha}$ by $\epsilon_\alpha$, etc.} in $\mathbb{E}_A$ and for each $\alpha < \aleph_1$
fix a promise $(k_\alpha,\Phi_\alpha)$  for the condition $p_{\alpha}$.
By passing to  an uncountable subset if necessary,
we may assume $\epsilon_\alpha
= \epsilon$, $h_\alpha = h$, $R_\alpha = R$, $k_\alpha = k$ for all $\alpha < \aleph_1$.
An application
of the $\Delta$-System Lemma (Lemma \ref{delta}) yields a finite set $Z \Subset A$ such that
$F_\alpha \cap F_\beta = Z$ for all $\alpha, \beta < \aleph_1$.
Since $Z$ is finite and
$D_{R, h}$ is countable, we can furthermore assume that for all $\alpha, \beta < \aleph_1$
if $a \in F_\alpha \cap F_\beta$ then $\psi_\alpha(a) = \psi_\beta(a)$.
Consider
\[
F = \bigcup_{\alpha < \aleph_1} F_\alpha.
\]
By \cite{farah2017calkin} there is a locally trivial embedding $\Theta: C^\ast(F) \to \mathcal{Q}(H)$.
For each $\alpha <\aleph_1$ fix a lift $\Theta_\alpha:
C^\ast(F_\alpha) \to \mathcal{B}(H)$ of $\Theta_{\restriction C^\ast(F_\alpha)}$.
Corollary \ref{voic} applied to $\Phi_\alpha$ and $\Theta_\alpha$
provides a faithful, essential, unital $\Phi'_\alpha: C^\ast(F_\alpha) \to \mathcal{B}(H)$
such that
\begin{enumerate}
\item $\Phi_\alpha'(a) - \Theta_\alpha(a) \in \mathcal{K}(H)$ for all $a \in F_\alpha$, hence
$\pi \circ \Phi_\alpha' = \Theta_{\restriction C^\ast(F_\alpha)}$,
\item $\Phi_\alpha'(a) h^+_\alpha = \Phi_\alpha(a) h^+_\alpha$ for all $a \in F_\alpha$.
\end{enumerate}
This entails that the pair $(k_\alpha, \Phi_\alpha')$ is still a promise for $p_\alpha$.
Hence, with no loss of generality, we can assume $\pi \circ \Phi_\alpha
= \Theta_{\restriction C^\ast(F_\alpha)}$ for every $\alpha < \aleph_1$.
This in particular implies that
\[
\Phi_\alpha (a) \sim_{K(H)} \Phi_\beta(a), \ \text{for all} \ a\in Z.
\]
Fix an arbitrary $\gamma >0$. We can assume that for all $\alpha, \beta \in \aleph_1$ and all $a \in F_\alpha \cap F_\beta$
\[
\lVert \Phi_\alpha(a) - \Phi_\beta(a) \rVert < \gamma.
\]
Indeed, start by fixing  $\delta < \aleph_1$. Then for each $\alpha < \aleph_1$ there is $P_\alpha \in
\mathcal{P}$ such that
\[
\lVert (\Phi_\alpha - \Phi_\delta)_{\restriction Z}(1 - P_\alpha) \rVert < \gamma/5
\]
and $R_\alpha \in \mathcal{P}$ such that
\[
\lVert (1 - R_\alpha) \Phi_{\alpha \restriction Z} P_\alpha \rVert < \gamma/5.
\]
We can assume $R_\alpha = R$ and $P_\alpha = P$ for all $\alpha < \aleph_1$
and since $R\mathcal{B}(H)P$ is finite-dimensional we can also require that
\[
\lVert R(\Phi_\alpha - \Phi_\beta)_{\restriction Z}P \rVert < \gamma/5
\]
for all $\alpha,\beta <\aleph_1.$ Thus, for $a\in Z,$ we have that: 
\begin{align*}
\lVert \Phi_\alpha(a) - \Phi_\beta(a) \rVert \le \lVert (\Phi_\alpha - \Phi_\beta)_{\restriction Z}P \rVert &+ \lVert (\Phi_\alpha - \Phi_\delta)_{\restriction Z}(1 - P) \rVert \\ &+
\lVert (\Phi_\beta - \Phi_\delta)_{\restriction Z}(1 - P) \rVert< \gamma.
\end{align*}
Since the choice of $\gamma$ was arbitrary, Lemma \ref{lemma:3} implies that we can pass to an uncountable subset in which any two conditions  $p_\alpha$ and $p_\beta$ 
are compatible.
\end{proof}

We quickly recall that Martin's Axiom, MA,  asserts that for every ccc poset $\bbP$ and every family $\cD$ of fewer than $2^{\aleph_0}$ 
dense open subsets there exists a filter in $\bbP$ intersecting all sets in $\cD$. 
\begin{corollary}\label{MA}
Assume MA. Then every \cstar-algebra with density character strictly less than $2^{\aleph_0}$ embeds into the Calkin algebra.
\end{corollary}

\begin{proof}
By Proposition \ref{fvlemma} it suffices to prove the statement for unital and simple \cstar-algebras.  For any unital and simple  \cstar-algebra $A$, the collection $\mathcal{D}$
of open, dense subsets of $\mathbb{E}_A$ (as defined prior to Proposition \ref{genemb}) has cardinality equal to the density character of $A$. Since the poset $\mathbb{E}_A$ is ccc, this implies that if the density character of~$A$ is strictly less than $2^{\aleph_0}$, then Martin's Axiom ensures the existence of a $\mathcal{D}$-generic filter for $\mathbb{E}_A$ and the corollary follows by Proposition \ref{genemb}.
\end{proof}

\section{Concluding remarks} \label{sctn5}
The Calkin algebra is a fascinating object and 
our result is the first step in what we believe is a very promising direction of its study. 
A further step would be to have a simpler forcing notion in place of  $\mathbb{E}_A$ defined in the course of the 
proof  of Theorem~\ref{T1}. 
This would allow for an analysis of the names for \cstar-subalgebras of $\cQ(H)$ and better control 
of the structure of $\cQ(H)$ in the extension. 
In particular, it would be a step towards proving that a given \cstar-algebra can be 
 `gently placed' into $\cQ(H)$ 
(cf. \cite[p. 17--18]{woodin1984discontinuous}). In this regard, we conjecture the following.
\begin{conjecture}
Let $A$ be an abelian and nonseparable $\mathrm{C}^*$-algebra.
If the density character of $A$ is greater than $2^{\aleph_0}$, then $\mathbb{E}_A$ forces that $A$ does not embed into $\ell_{\infty}/c_0.$
\end{conjecture}

A closely related issue is whether the embedding provided by $\mathbb{E}_A$ is `liftable' to $\cB(H)$. 
A unital,  injective   $^*$-homomorphism $\Phi\colon A\to \cQ(H)$ is naturally identified with an 
extension $0\to \mathcal{K}(H)\to D\to A\to 0$, where $D=\pi^{-1}[\Phi[A]]$ (see \cite[\S 2.5]{higsonroe}). 
An extension is \emph{trivial} if there exists a unital $^*$-homomorphism $\Psi\colon A\to D$
such that $\pi\circ \Psi=\id_A$.  If the \cstar-algebra $A$ is abelian, the triviality of this extension 
is equivalent to the image of $A$ having an abelian lift to $\cB(H)$. 
While our generic embeddings are easily seen to be  trivial when restricted to separable subalgebras
of $A$, we conjecture that this is not the case for the nonseparable ones, even in the abelian
case (see also \cite{bicekos} and \cite{obstructions}).

\begin{conjecture}Suppose that $A$ is a nonseparable abelian \cstar-algebra. 
The poset $\mathbb{E}_A$ forces that if $B$ is a nonseparable subalgebra of the image of $A$ under the generic embedding, then $B$ does not admit an abelian lift.
\end{conjecture}

We now propose   related  directions of study,  taking inspiration
from the commutative setting.

\subsection{The Question of Minimality of Generic Embeddings} \label{S.HE}
From the very beginnings of forcing, it has been known that a given partial ordering $E$
can be embedded into $\cP(\bbN)/\Fin$ by a ccc forcing. The simplest such forcing
notion was denoted  $\cH_E$ and studied in  \cite{Fa:Embedding} where it was proved that $\cH_E$ embeds
$E$ into $\cP(\bbN)/\Fin$ 
in a minimal way: If a cardinal $\kappa>2^{\aleph_0}$ is such that 
$E$ does not have a chain of order type  $\kappa$ or~$\kappa^*$,
then in the forcing 
extension  $\cP(\bbN)/\Fin$ does not have  chains of order type $\kappa$ or $\kappa^*$ (this is a consequence of \cite[Theorem~9.1]{Fa:Embedding}).
In addition, if $\min(\kappa,\lambda)>2^{\aleph_0}$ and $E$ does not have $(\kappa,\lambda)$-gaps\footnote{Given two cardinals $\kappa$ and $\lambda$, a \emph{$(\kappa, \lambda)$-gap}
in a poset $\mathbb{P}$ is composed by a strictly increasing sequence $\set{f_\alpha: \alpha < \kappa}\subseteq \mathbb{P}$ and
a strictly decreasing sequence $\set{g_\beta: \beta < \lambda} \subseteq \mathbb{P}$ such that
$f_\alpha < g_\beta$ for all $\alpha < \kappa$ and $\beta < \lambda$, and moreover such
that there is no $h \in \mathbb{P}$ greater than all $f_\alpha$'s and smaller than all $g_\beta$'s.}
then in the forcing 
extension by $\cH_E$ there are no   $(\kappa,\lambda)$-gaps (\cite[Theorem~9.2]{Fa:Embedding})
 in $\cP(\bbN)/\Fin$.  We do not know whether analogous results apply to $\mathbb{E}_A$ or some variant thereof. In the noncommutative setting, the following question is even more natural.

\begin{question} Consider the class $\bbE=\bbE(\cQ(H))$ 
of all \cstar-algebras that embed into the Calkin algebra. 
Can any notrivial closure properties of $\bbE$ be proved in ZFC? 
For example: 
\begin{enumerate} 
\item \label{1Q} Do $A\in \bbE$ and $B\in \bbE$ together imply $A\otimes B$ in $\bbE$ (take the spatial tensor product, 
or even the algebraic tensor product)? 
\item \label{2Q} If $A_n\in \bbE$ for $n\in \bbN$ and $A=\lim_n A_n$, is $A\in \bbE$? 
\end{enumerate}
\end{question}

We conjecture that the answers to both \eqref{1Q} and \eqref{2Q} are negative. 
The analogous class $\bbE_{\Fin}$ of all linear orderings that embed into $\cP(\bbN)/\Fin$ 
does not seem to have any nontrivial closure properties provable in ZFC. 
For example, it is relatively consistent with ZFC that there exists a linear ordering~$L$
and a partition $L=L_1\sqcup L_2$ such that $L_1\in \bbE_{\Fin}$ 
and $L_2\in \bbE_{\Fin}$ but $L\notin \bbE_{\Fin}$ (\cite[Proposition~1.4]{Fa:Embedding}).

\subsection{Complete embeddings} 
Given a forcing notion $\bbP$,  its subordering $\bbP_0$ is  a \emph{complete subordering} of~$\bbP$ if for every generic filter $G\subseteq\bbP_0$ one can define a forcing notion $\bbP/G$ such that 
$\bbP$ is forcing equivalent to the two-step iteration $\bbP_0*\bbP/G$ 
(for an intrinsic characterization of this relation see \cite[Definition III.3.65]{kunen}). 

A salient  property of the forcing notion $\cH_E$ (\S\ref{S.HE}) is that  
 $E\mapsto \cH_E$ is a covariant functor 
 from the category of partial orderings and order-isomorphic embeddings as maps 
into the category of forcing notions with complete embeddings as morphisms. 
This is a consequence of \cite[Proposition~4.2]{Fa:Embedding}, where  
the compatibility relation in $\cH_E$ has been shown to be `local' in the sense that 
the conditions $p$ and $q$ are compatible in 
$\cH_{\supp(p)\cup \supp(q)}$
if and only if they are 
compatible in $\cH_E$.

Analogous arguments show that the mapping $B\mapsto\mathbb{P}_B$ defined on Section~\ref{sctn3.1} is  a covariant functor from the category of Boolean algebras and injective homomorphisms into the category of ccc forcing notions with complete embeddings as morphisms. 
As a result, if $D$ is a Boolean subalgebra of $B$ and $G$ is $\mathbb{P}_D$-generic, then  forcing with the poset  $\mathbb{P}_B$ is equivalent to first forcing with $\mathbb{P}_D$ and then with~$\mathbb{P}_B/G.$

It is not difficult to prove that the association  $A\mapsto  \bbQD_A$
as in Proposition~\ref{qdt1} does not have this property, as
$\bbQD_{\C}$, naturally considered as a subordering of $\bbQD_{M_2(\bbC)}$, 
is not a complete subordering. 
More generally,  if  $m$ is a proper divisor of $n$ then 
the poset $\bbQD_{M_m(\C)}$ is not a complete subordering of $\bbQD_{M_n(\C)}$. 
We do not know whether there is an alternative 
definition of a functor $A\mapsto \bbQD_A$ that satisfies the conclusion of Proposition~\ref{qdt1}. 
The latter remark also  applies to the poset $\bbE_A$ given in Theorem~\ref{T1}.

\subsection{$2^{\aleph_0}$-universality} 
One line of research building on  Theorem \ref{indep} would be to understand which \cstar-algebras
of density character $2^{\aleph_0}$ embed into the Calkin algebra. Before discussing this matter,
we introduce a definition.
Given a cardinal $\lambda$, a \cstar-algebra $A$ is \emph{(injectively) $\lambda$-universal}
if it has density character $\lambda$ and all \cstar-algebras of density character $\lambda$ embed into~$A$.
By  \cite[Theorem~2.3 and Remark 2.10]{JunPis}, 
there is no $\kappa$-universal \cstar-algebra in any density character $\kappa<2^{\aleph_0}$ 
The results in \cite{farah2017calkin} entail that the $2^{\aleph_0}$-universality of the Calkin algebra is independent
from ZFC. On the one hand
CH implies that $\mathcal{Q}(H)$ is $2^{\aleph_0}$-universal.
Conversely, the Proper Forcing Axiom implies that $\cQ(H)$
is not $2^{\aleph_0}$-universal because some abelian \cstar-algebras of density $2^{\aleph_0}$ do not embed into it (see \cite[Corollary 5.3.14 and Theorem 5.3.15]{V.PhDThesis}); see also
Theorem~\ref{indep}).
Can the Calkin algebra be $2^{\aleph_0}$-universal even when the Continuum Hypothesis fails? 
The analogous fact for $\cP(\bbN)/\Fin$ and linear orders, namely that there is a model of ZFC
where CH fails and all linear orders of size $2^{\aleph_0}$
 embed into $\cP(\bbN)/\Fin$, has been
proved in \cite{laver1979linear} (see also \cite{baum} for the generalization
to Boolean algebras). We do not know whether these
techniques can be generalized to provide a model in which CH fails and 
the  Calkin algebra is a $2^{\aleph_0}$-universal \cstar-algebra, but the fact that $\bbE_A$ has property K 
is a step (possibly small) towards such a model. A poset with property K is \emph{productively ccc}, 
in the sense that its product with any ccc poset is still ccc. 
A salient feature of the forcing iterations used in both  \cite{laver1979linear} and   \cite{baum}
is that they are not `freezing' any gaps in 
$\bbN^\bbN/\Fin$ and $\cP(\bbN)/\Fin$.
(A poset $\bbP$ \emph{freezes} a gap  
if it cannot be split in a further forcing extension without collapsing~$\aleph_1$.)

\begin{lemma} For any \cstar-algebra $A$, the poset  $\bbE_A$ cannot freeze any gaps in $\cP(\bbN)/\Fin$. 
\end{lemma} 

\begin{proof} 
Every gap in $\cP(\bbN)/\Fin$ or $\bbN^\bbN/\Fin$  that can be split without collapsing $\aleph_1$ can be split by a ccc forcing. This is well-known result of Kunen (\cite{Ku:Gaps}) 
not so easy to find in the literature.\footnote{See e.g., \cite[Fact on p. 76]{tofa}. 
It is not difficult to see that a `Suslin gap' as in \cite[Definition~9.4]{tofa} can be split by a natural ccc forcing whose conditions are finite 
$K_0$-homogeneous sets.}
Therefore if a gap can be split by a ccc forcing $\bbP$, 
then a poset which  freezes it destroys the ccc-ness of $\bbP$. 
But $\bbE_A$ has property K, and is therefore productively ccc. 
\end{proof}

While the gap spectra of $\cP(\bbN)/\Fin$ and $\bbN^\bbN/\Fin$ are closely related, the gap spectrum of the poset of projections in the Calkin algebra is  more 
complicated. The following proposition was proved, but not stated, in \cite{Za-Av:Gaps}, 
and we include a proof for reader's convenience.

\begin{theorem}  Martin's Axiom implies that the poset of projections in the 
 Calkin algebra contains a $(2^{\aleph_0}, 2^{\aleph_0})$-gap which cannot be frozen. 
\end{theorem} 

\begin{proof} By \cite[Theorem~4]{Za-Av:Gaps}, 
 there exists (in ZFC)  a  gap in this poset whose sides are analytic and $\sigma$-directed.
 This gap cannot be frozen, and Martin's Axiom is used only to `linearize' it.  
 By the  discussion following 
\cite[Corollary~2]{Za-Av:Gaps}, each of the sides of this gap is 
Tukey equivalent to the ideal of Lebesgue measure  zero sets ordered by the inclusion. 
Since  the additivity of the Lebesgue measure can be increased by a ccc poset
(\cite[Lemma~III.3.28]{kunen}), 
Martin's Axiom implies that this gap contains an $(2^{\aleph_0}, 2^{\aleph_0})$-gap
and that any further ccc forcing that increases the additivity of the Lebesgue measure
will split the gap. 
\end{proof}

\subsection*{Acknowledgments}
The authors would like to thank Marton Elekes for a helpful remark, Alessandro Vignati for his useful feedback on the earlier drafts of this paper, as well as the anonymous  referees for suggesting several improvements.

\bibliographystyle{amsalpha}
	\bibliography{Bibliography}
\Addresses

\end{document}